\documentclass[final]{siamltex}
 
\usepackage{graphicx}
\usepackage{amsfonts}
\usepackage{amsmath}
\usepackage{amssymb}
\usepackage{bm}
\usepackage{amsbsy}
\usepackage{algorithmic}
\usepackage{algorithm}
\usepackage{subfigure}
\usepackage{listings}
\usepackage{tikz}
\usepackage{stackrel}
\usepackage{url}
\usepackage{xspace}
\usepackage{color}            %

\usepackage[retainorgcmds]{IEEEtrantools}

\usepackage[T1]{fontenc}
\usepackage[ansinew]{inputenc}
\usepackage{lmodern}

\usepackage{ifpdf}
\ifpdf
  \usepackage{epstopdf} 
\fi

\usepackage[letterpaper,centering]{geometry}
\usepackage[colorlinks=true, citecolor=blue, filecolor=black,linkcolor=red, urlcolor=blue, hypertexnames=false]{hyperref}

\usepackage{todonotes}

\newcommand{\highlight}[1]{{\textcolor{black}{#1}}}
\newcommand{\highlightTwo}[1]{{\textcolor{black}{#1}}}

\newcommand{\Uh}{\widehat{U}}

\newcommand{\Ac}{{\cal A}}

\newcommand{\Cc}{{\cal C}}
\newcommand{\Fc}{{\cal F}}

\newcommand{\Jc}{{\cal J}}
\newcommand{\Lc}{{\cal L}}
\newcommand{\Mc}{{\cal M}}
\newcommand{\Nc}{{\cal N}}
\newcommand{\Rc}{{\cal R}}

\newcommand{\Uc}{{\cal U}}
\newcommand{\Ex}{\mathbb{E}}

\newcommand{\Var}{{\mathbb{V}{\rm ar}}}
\newcommand{\vare}{\varepsilon}

\def\zal2{z_{\alpha/2}}

\newcommand{\R}{{\mathbb{R}}}

\newcommand{\Gy}{\Gamma_{\rm y}}

\newcommand{\Gobs}{\Gamma_{\rm obs}}
\newcommand{\Gpr}{\Gamma_{\rm pr}}
\newcommand{\Gauss}{\Nc}

\newcommand{\Gprs}{S_{\rm pr}} 
\newcommand{\Gprst}{S_{\rm pr}^\top} 

\newcommand{\Gposhs}{\widehat{S}_{\rm pos}}
\newcommand{\Gposs}{S_{\rm pos}}
\newcommand{\Gys}{S_{\rm y}}
\newcommand{\Hhs}{S_{ \widehat{H}}}
\newcommand{\Gobss}{S_{\rm obs}}

\newcommand{\Gpos}{\Gamma_{\rm pos}}
\newcommand{\nupos}{\nu_{\rm pos}}
\newcommand{\nuposh}{\widehat{\nu}_{\rm pos}}
\newcommand{\Gposh}{\widehat{\Gamma}_{\rm pos}}
\newcommand{\mupos}{\mu_{\rm pos}(y)}
\newcommand{\muposr}{\mu_{\rm pos}^{(r)}(y)}
\newcommand{\muposhr}{\widehat{\mu}_{\rm pos}^{(r)}(y)}
\newcommand{\df}{d_{\Fc}}
\newcommand{\trace}{\mathrm{tr}}
\newtheorem{remark}{Remark}

\graphicspath{{figures_paper/}}

\title{Optimal low-rank approximations of Bayesian linear inverse problems }

\author{Alessio Spantini\footnotemark[1] \and Antti
  Solonen{\footnotemark[1] \footnotemark[4]} \and Tiangang Cui\footnotemark[1] \and James
  Martin\footnotemark[2] \and Luis Tenorio\footnotemark[3] \and Youssef
  Marzouk\footnotemark[1]}

\begin{document}

\maketitle

\renewcommand{\thefootnote}{\fnsymbol{footnote}}

\footnotetext[1]{Department of Aeronautics and Astronautics, Massachusetts Institute of Technology, Cambridge, MA
02139, USA, \texttt{\{spantini, solonen, tcui, ymarz\}@mit.edu}.}

\footnotetext[2]{Institute for Computational Engineering and Sciences,
  University of Texas at Austin, Austin, TX 78712 USA, \texttt{jmartin@ices.utexas.edu}.}

\footnotetext[3]{Applied Mathematics and Statistics, Colorado School of Mines, Golden, CO 80401, USA, \texttt{ltenorio@mines.edu}.}

\footnotetext[4]{Lappeenranta University of Technology, Department of Mathematics and Physics, 53851 Lappeenranta, Finland.}

\renewcommand{\thefootnote}{\arabic{footnote}}

\begin{abstract} 
  In the Bayesian approach to inverse problems, data are often
  informative, relative to the prior, only on a low-dimensional
  subspace of the parameter space. Significant computational savings
  can be achieved by using this subspace to characterize and
  approximate the posterior distribution of the parameters.
  We first investigate approximation of the posterior covariance
  matrix as a low-rank update of the prior covariance matrix. We prove
  optimality of a particular update, based on the leading
  eigendirections of the matrix pencil defined by the Hessian of the negative 
  log-likelihood and the prior precision, for a broad class of loss
  functions. This class includes the F\"{o}rstner metric for symmetric
  positive definite matrices, as well as the Kullback-Leibler
  divergence and the Hellinger distance between the associated
  distributions.
  We also propose two fast approximations of the posterior mean and
  prove their optimality with respect to a weighted Bayes risk under
  squared-error loss. 
  \highlightTwo{
  These approximations are deployed in an
   offline-online manner, where a more costly but data-independent
   offline calculation is followed by fast online evaluations.  
  As a result, these approximations are particularly useful
  when repeated posterior mean evaluations are required for multiple
  data sets.
  }
  We demonstrate our theoretical results with several numerical
  examples, including high-dimensional X-ray tomography and an inverse
  heat conduction problem. In both of these examples, the intrinsic
  low-dimensional structure of the inference problem can be exploited
  while producing results that are essentially indistinguishable from
  solutions computed in the full space.

 \end{abstract}


\pagestyle{myheadings}
\thispagestyle{plain}

\markboth{\MakeUppercase{Spantini et al.}}{\MakeUppercase{Optimal low-rank approximations}}

\section{Introduction}
\label{s:intro}

In the Bayesian approach to inverse problems, the parameters of interest are treated as random variables, endowed with a prior probability distribution that encodes information available before any data are observed. Observations are modeled by their joint probability distribution conditioned on the parameters of interest, which defines the likelihood function and incorporates the forward model and a stochastic description of measurement or model errors. The prior and likelihood then combine to yield a probability distribution for the parameters conditioned on the observations, i.e., the posterior distribution. While this formulation is quite general, essential features of inverse problems bring additional structure to the Bayesian update. The prior distribution often encodes some kind of smoothness or correlation among the inversion parameters; observations typically are finite, few in number, and corrupted by noise; and the observations are indirect, related to the inversion parameters by the action of a forward operator that destroys some information. A key consequence of these features is that the data may be informative, relative to the prior, only on a \textit{low-dimensional subspace} of the entire parameter space. Identifying and exploiting this subspace---to design approximations of the posterior distribution and related Bayes estimators---can lead to substantial computational savings.

In this paper we investigate approximation methods for finite-dimensional Bayesian linear inverse problems with Gaussian measurement and prior distributions. We characterize approximations of the posterior distribution that are structure-exploiting and that are \textit{optimal} in a sense to be defined below. Since the posterior distribution is Gaussian, it is completely determined by its mean and covariance. We therefore focus on approximations of these posterior characteristics. Optimal approximations will reduce computation and storage requirements for high-dimensional inverse problems, and will also enable fast computation of the posterior mean in a many-query setting.

We consider approximations of the posterior covariance matrix in the form of low-rank negative updates of the prior covariance matrix. This class of approximations exploits the structure of the prior-to-posterior update, \highlight{and also arises naturally in Kalman filtering techniques (e.g., \cite{auvinen2009large,auvinen2010variational,solonen2012variational})}; the challenge is to find an \textit{optimal} update within this class, and to define in what sense it is optimal.  We will argue that a suitable loss function with which to define optimality is the F\"{o}rstner metric \cite{forstner2003metric} for symmetric positive definite matrices, and will show that this metric generalizes to a broader class of loss functions that emphasize relative differences in covariance. We will derive the optimal low-rank update for this entire class of loss functions. In particular, we will show that the prior covariance matrix should be updated along the leading generalized eigenvectors of the pencil $(H,\Gpr^{-1})$ defined by the Hessian of the negative log-likelihood and the prior precision matrix.
If we assume exact knowledge of the posterior mean, then our results extend to optimality statements between distributions (e.g., optimality in Kullback-Leibler divergence and in Hellinger distance).
The form of this low-rank update of the prior is not new \cite{bui2012extreme,bui2013computational,flath2011fast,martin2012stochastic}, but previous work has not shown whether---and if so, in exactly what sense---it yields optimal approximations of the posterior. A key contribution of this paper is to establish and explain such optimality.

\highlight{Properties of the generalized eigenpairs of $(H,\Gpr^{-1})$ and related matrix pencils have been studied previously in the literature, especially in the context of classical regularization techniques for linear inverse problems\footnote{In the framework of Tikhonov regularization \cite{tikhonov1963solution}, the regularized estimate coincides with the posterior mean of the Bayesian linear model we consider here, provided that the prior covariance matrix is chosen appropriately.} \cite{van1976generalizing,paige1981towards,hansen1989regularization,hanke1993regularization,dykes2014simplified}.  The joint action of the log-likelihood Hessian and the prior precision matrix has also been used in related regularization methods \cite{calvetti2005priorconditioners,calvetti2007preconditioned, calvetti2012left, homa2013bayesian}. However, these efforts have not been concerned with the posterior covariance matrix or with its optimal approximation, since this matrix is a property of the Bayesian approach to inversion.}
One often justifies the assumption that the posterior mean is exactly known by arguing that it can easily be computed as the solution of a \highlight{regularized least-squares problem
\cite{hestenes1952methods,paige1982algorithm,akccelik2006parallel,meng2014lsrn,barrett1994templates}};
 indeed, evaluation of the posterior mean to machine precision is now feasible even for million-dimensional parameter spaces \cite{bui2012extreme}. If, however, one needs multiple evaluations of the posterior mean for different realizations of the data (e.g., in an online inference context), then solving a linear system to determine the posterior mean may not be the most efficient strategy. A second goal of our paper is to address this problem. We will propose two computationally efficient approximations of the posterior mean based on: \textit{(i)} evaluating a low-rank affine function of the data; 
or \textit{(ii)} using a low-rank update of the prior covariance matrix in the exact formula for the posterior mean. The optimal approximation in each case is defined as the minimizer of the Bayes risk for a squared-error loss weighted by the posterior precision matrix. We provide explicit formulas for these optimal approximations and show that they can be computed by exploiting the optimal posterior covariance approximation described above. Thus, given a new set of data, computing an optimal approximation of the posterior mean becomes a computationally trivial task. 

Low-rank approximations of the posterior mean that minimize the Bayes risk for squared-error loss 
have been proposed in \cite{chung2014efficient,chung2014optimal,chung2012optimal,chung2011designing,chung2013computing} for a general
non-Gaussian case. 
Here, instead we develop analytical results for squared-error loss weighted by the posterior precision matrix. This choice of norm reflects the idea that approximation errors in directions of low posterior variance should be penalized more strongly than errors in high-variance directions, as we do not want the approximate posterior mean to fall outside the bulk of the posterior probability distribution. Remarkably, in this case, the optimal approximation only requires the leading eigenvectors and eigenvalues of a single eigenvalue problem. This is the same eigenvalue problem we solve to obtain an optimal approximation of the posterior covariance matrix, and thus we can efficiently obtain both approximations at the same time.

While the efficient solution of large-scale linear-Gaussian Bayesian inverse problems is of standalone interest \cite{flath2011fast}, optimal approximations of Gaussian posteriors are also a building block for the solution of nonlinear Bayesian inverse problems. For example, the stochastic Newton Markov chain Monte Carlo (MCMC) method \cite{martin2012stochastic} uses Gaussian proposals derived from local linearizations of a nonlinear forward model; the parameters of each Gaussian proposal are computed using the optimal approximations analyzed in this paper. To tackle even larger nonlinear inverse problems, \cite{bui2012extreme} uses a Laplace approximation of the posterior distribution wherein the Hessian at the mode of the log-posterior density is itself approximated using the present approach. Similarly, approximations of local Gaussians can facilitate the construction of a nonstationary Gaussian process whose mean directly approximates the posterior density \cite{bui2012adaptive}.  Alternatively, \cite{cui2014likelihood} combines data-informed directions derived from local linearizations of the forward model---a direct extension of the posterior covariance approximations described in the present work---to create a global \textit{data-informed subspace}. A computationally efficient approximation of the posterior distribution is then obtained by restricting MCMC to this subspace and treating complementary directions analytically. Moving from the finite to the infinite-dimensional setting, the same global data-informed subspace is used to drive efficient dimension-independent posterior sampling for inverse problems in \cite{CLM_2014}.

Earlier work on dimension reduction for Bayesian inverse problems used the Karhunen-Lo\`{e}ve expansion of the prior distribution \cite{marzouk2009dimensionality, Li:2006} to describe the parameters of interest. To reduce dimension, this expansion is truncated; this step renders both the prior and posterior distributions singular---i.e., collapsed onto the prior mean---in the neglected directions. Avoiding large truncation errors then requires that the prior distribution impose significant smoothness on the parameters, so that the spectrum of the prior covariance kernel decays quickly. In practice, this requirement restricts the choice of priors. Moreover, this approach relies entirely on properties of the prior distribution and does not incorporate the influence of the forward operator or the observational errors. Alternatively, \cite{lieberman2010parameter} constructs a reduced basis for the parameter space via greedy model-constrained sampling, but this approach can also fail to capture posterior variability in directions uninformed by the data. Both of these earlier approaches seek reduction in the overall description of the parameters. This notion differs fundamentally from the dimension reduction technique advocated in this paper, where low-dimensional structure is sought in the \textit{change} from prior to posterior. 

The rest of this paper is organized as follows. In Section \ref{sec:theory} we introduce the posterior covariance approximation problem and derive the optimal prior-to-posterior update with respect to a broad class of loss functions. The structure of the optimal posterior covariance matrix approximation is examined in Section \ref{sec:interpretation}.  Several interpretations are given in this section, including an equivalent reformulation of the covariance approximation problem as an optimal \textit{projection} of the likelihood function onto a lower dimensional subspace. In Section \ref{sec:mean} we characterize optimal approximations of the posterior mean. In Section \ref{s:examples} we provide several numerical examples. Section \ref{s:conclusions} offers concluding remarks. 
Appendix \ref{sec:proofs} collects proofs of many of the theorems stated throughout the paper, 
along with additional technical results.

\section{Optimal approximation of the posterior covariance matrix} \label{sec:theory}

Consider the Bayesian linear model defined by a Gaussian likelihood and a Gaussian
prior with a non-singular covariance matrix $\Gpr \succ 0$  and, without loss of generality,
zero mean:
\begin{equation}\label{eq:GBLM}
y\mid x \sim \Gauss(Gx,\,\Gobs),\quad x\sim \Gauss(0,\,\Gpr).
\end{equation}
\highlight{
Here $x$ represents the parameters to be inferred, $G$ is the linear
forward operator, and $y$ are the observations, with $\Gobs \succ
0$. The statistical model \eqref{eq:GBLM} also follows from:
\begin{equation*}
  y = G x + \varepsilon
\end{equation*}
where $\varepsilon \sim \mathcal{N}(0,\Gobs)$  is independent of $x$.
}
It is easy to see that the posterior distribution is again Gaussian
(see, e.g., \cite{carlin2009bayesian}):\,
\(
x \mid  y \sim \Gauss(\mupos, \Gpos),
\)
with mean and covariance matrix given by
\begin{equation}\label{eq:postmoms}
\mupos = \Gpos \,G^\top \Gobs^{-1}\,y\quad \text{and} \quad \Gpos = \left(H + \Gpr^{-1}\right)^{-1},
\end{equation}
where 
\begin{equation}\label{eq:hessian}
H= G^\top\Gobs^{-1}G
\end{equation}
is the Hessian of the negative log-likelihood (i.e., the Fisher
information matrix). Since the posterior is Gaussian, the posterior
mean coincides with the posterior mode: $ \mupos = {\arg\, \max}_x\,
\pi_{\rm pos}(x;y)$, where $\pi_{\rm pos}$ is the posterior
density. Note that the posterior covariance matrix does not depend on
the data.

\subsection{Defining the approximation class}
We will seek an approximation, $\Gposh$, of the posterior covariance
matrix that is optimal in a class of matrices to be defined
shortly. As we can see from \eqref{eq:postmoms}, the posterior
precision matrix $\Gpos^{-1}$ is a non-negative update of the prior
precision matrix
\highlight{
 $\Gpr^{-1}$: $\Gpos^{-1} = \Gpr^{-1} + ZZ^\top$,
where $ZZ^\top = H$.}
Similarly, using Woodbury's identity we can write
$\Gpos$ as a non-positive update of $\Gpr$:
\(
\Gpos = \Gpr - KK^\top,
\)
where
\(
KK^\top = \Gpr \,G^\top\Gy^{-1} G\, \Gpr
\) 
and $\Gy = \Gobs + G\,\Gpr \,G^\top$ is the covariance matrix of the marginal
distribution of $y$ \cite{kaipio2005statistical}. This update of $\Gpr$ is negative semidefinite
because the data add information: the posterior variance in any direction is always smaller than
the corresponding prior variance. 
 Moreover, the update is usually low rank
for exactly the reasons described in the introduction: there are
directions in the parameter space along which the data are not very
informative, relative to the prior. For instance, $H$ might have a
quickly decaying spectrum \cite{bui2012analysis, schillings2014scaling}. 
Note, however, that $\Gpos$ itself might \textit{not} be low-rank. Low-rank
structure, if any, lies in the update of $\Gpr$ that yields $\Gpos$.
Hence, a natural class of matrices for approximating $\Gpos$ is the set
of negative semi-definite updates of $\Gpr$, with a fixed maximum
rank, that lead to positive definite matrices:
\begin{equation} \label{eq:class_M}
    {\Mc}_r=\left\{ \Gpr-KK^\top\succ0  :\mathrm{rank}(K)\le r  \right\}.
\end{equation}
This class of approximations of the posterior covariance matrix takes
advantage of the structure of the prior-to-posterior update.

\subsection{Loss functions}
Optimality statements regarding the approximation of a covariance
matrix require an appropriate notion of distance between symmetric
positive definite (SPD) matrices. We shall use the metric introduced
by F\"{o}rstner and Moonen \cite{forstner2003metric}, which is derived
from a canonical invariant metric on the cone of real SPD matrices
and is defined as follows: the F\"{o}rstner distance, $\df(A,B)$,
between a pair of SPD matrices, $A$ and $B$, is given by
\[
\df^2(A,B)=\trace \left [ \,\ln^2 (\,A^{-1/2}BA^{-1/2}\,)\, \right ]=\sum_i \ln^2(\sigma_i), 
\]
where $(\sigma_i)$ is the sequence of generalized eigenvalues of the pencil
$(A,B)$.  The F\"{o}rstner metric satisfies the following important
invariance properties:
\begin{equation}\label{eq:dfprop}
	\quad \df(A,\, B )=\df(A^{-1},\, B^{-1} ),\quad \mbox{and}\quad \df(A,\, B )=\df(M A M^\top,\, M B M^\top )
\end{equation}
for any nonsingular matrix $M$. 
Moreover, $\df$ treats under- and over-approximations similarly in the sense that
\(
\df(\Gpos,\,\alpha\Gposh)\to \infty
\) 
as $\alpha\to 0$ and as $\alpha\to \infty$.\footnote{This behavior is
  shared by Stein's loss function, which has been proposed to assess
  estimates of a covariance matrix \cite{james1961estimation}. Stein's
  loss function is just the Kullback-Leibler distance between two
  Gaussian distributions with the same mean (see
  \eqref{eq:KL_formula}), but it is not a metric for SPD matrices.}
Note that the metric induced by the Frobenius norm does not
satisfy any of the aforementioned invariance properties. In addition,
it penalizes under- and over-estimation differently.

We will show that our posterior covariance matrix approximation is
optimal not only in terms of the F\"{o}rstner metric, but also in
terms of the following more general class, $\Lc$, of loss functions
for SPD matrices.  
\medskip
\begin{definition}[Loss functions]
  \label{def:lossfunc}
  The class $\Lc$ is defined as the collection of functions of the
  form
  \begin{equation} \label{eq:form_loss_fun}
    L(A , B ) = \sum_{i=1}^n f(\sigma_i) ,
  \end{equation}
  where $A$ and $B$ are SPD matrices, $(\sigma_i)$ are the generalized
  eigenvalues of the pencil $(A,B)$, and
  \begin{equation}\label{eq:Uclass}
    f\in \Uc = \{ g\in\Cc^1(\R_+):g'(x)(1-x)<0\,\,\mathrm{for}\,x\neq
    1, \ \mathrm{and} \ \lim_{x\rightarrow\infty} g(x) = \infty  \}.
  \end{equation}
\end{definition}
Elements of $\Uc$ are differentiable real-valued functions
defined on the positive axis that decrease on $x<1$,
increase on $x>1$, and tend to infinity as $x\rightarrow
\infty$. The squared F\"{o}rstner metric belongs to the class of loss
functions defined by \eqref{eq:form_loss_fun}, whereas the
distance induced by the Frobenius norm does not.

Lemma \ref{lemma:equivalence}, whose proof can be found in Appendix
\ref{sec:proofs}, justifies the importance of the class $\Lc$.  In
particular, it shows that optimality of the covariance matrix
approximation with respect to any loss function in $\Lc$ leads to an
optimal approximation of the posterior distribution using 
a Gaussian (with the same mean)
in terms of other familiar criteria
used to compare probability measures, such as the Hellinger distance
and the Kullback-Leibler (K-L) divergence \cite{pardo2005statistical}.
More precisely, we have the following result: \medskip

\begin{lemma}[Equivalence of approximations] \label{lemma:equivalence}
  If $L\in \Lc$, then a matrix $\Gposh \in {\Mc}_r$
  minimizes the Hellinger distance and the K-L divergence between
  $\Gauss(\mupos,\Gpos)$ and the approximation $\Gauss(\mupos,\Gposh)$ iff it
  minimizes $L(\, \Gpos ,\Gposh\,)$.
\end{lemma}

\medskip
\begin{remark}{\rm 
We note that neither the Hellinger distance nor the K-L divergence between the distributions
$\Gauss(\mupos,\Gpos)$ and $\Gauss(\mupos,\Gposh)$  
depends on the data $y$. Optimality in distribution does not necessarily hold when the 
posterior means are different.}
\end{remark}

\subsection{Optimality results} 
We are now in a position to state one of the main results of the
paper. For a proof see Appendix \ref{sec:proofs}.
\medskip
\begin{theorem} [Optimal posterior covariance approximation] \label{Thm:main_th} 
\highlight{
Let $(\delta_i^2,\widehat{w}_i)$ be the generalized 
  eigenvalue-eigenvector pairs of the pencil:
\begin{equation}\label{eq:precondH}
 (H,\Gpr^{-1}),
\end{equation}
with the ordering $\delta_i^2\ge\delta_{i+1}^2$, and $H=G^\top \Gobs^{-1} G$ as in \eqref{eq:hessian}. Let $L$ be a loss function in the class $\Lc$ defined in \eqref{eq:form_loss_fun}.
}
 Then:
\begin{itemize}
\item[(i)] A minimizer, $\Gposh$, of the loss, $L$, between $\Gpos$ and an element of ${\Mc}_r$ is given by:
	\highlight{
  \begin{equation} \label{minimizer_theorem}
    \Gposh=\Gpr-KK^\top, \quad KK^\top=\sum_{i=1}^r \delta_i^2 \left( 1 + \delta_i^2 \right)^{-1} \widehat{w}_i \widehat{w}_i^\top.
  \end{equation}
  }
  The corresponding minimum  loss is given by:
  \begin{equation} \label{eq:error_estimate}
    L( \Gposh ,\Gpos )= f\left( 1 \right) r+\sum_{i>r}  f(\,{1}/{(1+\delta_i^2)}\,).
  \end{equation}    
\item[(ii)] The minimizer \eqref{minimizer_theorem} is unique if the
  first $r$ 
   eigenvalues of $(H,\Gpr^{-1})$ are different.
\end{itemize}
\end{theorem}
\medskip

Theorem \ref{Thm:main_th} provides a way to compute the best
approximation of $\Gpos$ by matrices in ${\Mc}_r$:
\highlight
{
 it is just a matter
of computing the eigenpairs corresponding to the decreasing sequence
of eigenvalues of the pencil  $(H,\Gpr^{-1})$ until a stopping criterion is
satisfied.  This criterion can be based on the minimum loss
\eqref{eq:error_estimate}. Notice that the minimum loss is a function of the generalized eigenvalues $(\delta_i^2)_{i\ge r}$ 
 that have not been computed. This is quite common in numerical linear algebra 
 (e.g., error in the truncated SVD \cite{eckart1936approximation, golub2012matrix}). However, since the eigenvalues $(\delta_i^2)$ are
 computed in a decreasing order, the minimum loss can be easily bounded.
 }

The generalized eigenvectors, $\widehat{w}_i$, are orthogonal with respect
to the inner product induced by the prior precision matrix, and they
maximize the Rayleigh ratio,
\begin{equation*}
  \widehat{\Rc}(z)  =  \frac{z^\top H z }{ z^\top \Gpr^{-1} z},                  
\end{equation*}
over subspaces of the form $\widehat{\mathcal{W}}_i={\rm span}^\perp
(\widehat{w}_j)_{j<i}$.  Intuitively, the vectors $\widehat{w}_i$
associated with generalized eigenvalues greater than one correspond to
directions in the parameter space (or subspaces thereof) where the
curvature of the log-posterior density is constrained more by the
log-likelihood than by the prior.

\subsection{Computing eigenpairs of $(H,\Gpr^{-1})$}
\highlight{If a square root factorization of the prior covariance matrix $\Gpr=\Gprs\Gprs^\top$ is available, then
  the Hermitian generalized eigenvalue problem can be reduced to a standard one: find the eigenpairs, $(\delta_i^2,w_i)$,
  of $\Gprs^{\top}H\Gprs$, and transform the resulting eigenvectors according to $w_i \mapsto \Gprs w_i$
  \cite[Section~5.2]{bai2000templates}.  An analogous transformation is also possible when a square root factorization
  of $\Gpr^{-1}$ is available.  Notice that only the actions of $\Gprs$ and $\Gprs^\top$ on a vector are required.  For
  instance, evaluating the action of $\Gprs$ might involve the solution of an elliptic PDE \cite{lindgren2011explicit}.
  There are numerous examples of priors for which a decomposition $\Gpr=\Gprs \Gprs^\top$ is readily available, e.g.,
  \cite{wood1994simulation,dietrich1997fast,lindgren2011explicit,yue2010nonstationary,stuart2010inverse}.  Either direct
  methods or, more often, matrix-free algorithms (e.g., Lanczos iteration or its block version 
  \cite{lanczos1950iteration,paige1972computational,cullum1974block,golub1977block}) 
  can be used to solve the standard Hermitian
  eigenvalue problem \cite[Section~4]{bai2000templates}. Reference implementations of these algorithms are available in
  ARPACK \cite{lehoucq1998arpack}.
  \highlightTwo{We note that the Lanzos iteration comes with a rich literature 
  on error analysis (e.g., \cite{kaniel1966estimates,paige1971computation,paige1976error, kahan1978far,saad1980rates,golub2012matrix}).
  Alternatively, one can use randomized methods \cite{halko2011finding}, which offer the
  advantage of parallelism (asynchronous computations) and robustness over standard Lanczos methods \cite{bui2012extreme}.}
  If a square root factorization of $\Gpr$ is not available, but it is possible to solve linear systems with
  $\Gpr^{-1}$, we can use a Lanczos method for generalized Hermitian eigenvalue problems
  \cite[Section~5.5]{bai2000templates} where a Krylov basis orthogonal with respect to the inner product induced by
  $\Gpr^{-1}$ is maintained. Again, ARPACK provides an efficient
  implementation of these solvers.  When accurately solving
  linear systems with $\Gpr^{-1}$ is a difficult task, we refer the reader to alternative algorithms proposed in
  \cite{sameh1982trace} and \cite{golub2002inverse}.}

\medskip
\begin{remark}
\highlight{
{\rm If a factorization $\Gpr=\Gprs \Gprs^\top$ is available, then it is straightforward to obtain an expression for a non-symmetric square root of the optimal approximation of $\Gpos$ \eqref{minimizer_theorem} as in \cite{bui2013computational}:
 \begin{equation}
  \Gposhs =\Gprs \left( \sum_{i=1}^r  \left[ \left(1+\delta_i^2\right)^{-1/2}-1 \right] w_i w_i^\top  + I \right)
 \end{equation}
 such that $\Gposh=\Gposhs \,\Gposhs^\top$ and $w_i=\Gprs^{-1}\widehat{w}_i$. 
This expression can be used  to efficiently sample 
from the approximate posterior distribution $\Gauss(\mupos ,\Gposh)$ (e.g., \cite{flath2011fast,martin2012stochastic}).   
\highlightTwo{Alternative techniques for sampling from high-dimensional Gaussian distributions can be found, for instance, in \cite{parker2012sampling,fox2014convergence}.}}
}
\end{remark}

\section{Properties of the optimal covariance approximation} \label{sec:interpretation}

Now we discuss several implications of the optimal approximation of $\Gpos$ introduced in the previous section. 
We start by describing the relationship between this approximation and the directions of greatest relative reduction of prior variance. Then we interpret the covariance approximation as the result of projecting the likelihood function onto a ``data-informed'' subspace. Finally, we contrast the present approach with several other approximation strategies: using the Frobenius norm as a loss function for the covariance matrix approximation, or developing low-rank approximations based on prior or Hessian information alone. We conclude by drawing the connections with the BFGS Kalman filter update.

\subsection{Interpretation of the eigendirections}

Thanks to the particular structure of loss functions in $\Lc$, the problem of approximating $\Gpos$ is equivalent to that of approximating $\Gpos^{-1}$.  Yet the form of the optimal approximation of $\Gpos^{-1}$ is important, as it explicitly describes the directions that control the ratio of posterior to prior variance. 
The following corollary to Theorem \ref{Thm:main_th} characterizes these directions. The proof is in Appendix~\ref{sec:proofs}.

\medskip
\begin{corollary}[Optimal posterior precision approximation] \label{cor:Minv}
\highlight{
Let $(\delta_i^2,\widehat{w}_i)$ and $L\in\Lc$ be defined as in Theorem \ref{Thm:main_th}. Then: 
}
     \begin{itemize}
\item[(i)] A minimizer of $L(B,\Gpos^{-1})$ for
    \begin{equation} \label{eq:class_Minv}
  B\in  {\Mc}_r^{-1} := \left\{ \Gpr^{-1}+JJ^\top  :\mathrm{rank}(J) \le r  \right\}
     \end{equation}
is given by
\highlight{
    \begin{equation}\label{eq:minthminv}
\Gposh^{-1} = \Gpr^{-1} + UU^\top,\quad UU^\top = \sum_{i=1}^r \delta_i^2 \widetilde{w}_i\widetilde{w}_i^\top,\quad \widetilde{w}_i=\Gpr^{-1} \widehat{w}_i .
    \end{equation}
The minimizer \eqref{eq:minthminv} is unique if the first $r$ 
eigenvalues of $(H,\Gpr^{-1})$ are different.
}
\item[(ii)]The optimal posterior precision matrix \eqref{eq:minthminv} is precisely the 
inverse of the optimal posterior covariance matrix \eqref{minimizer_theorem}.
 \item[(iii)] The vectors $\widetilde{w}_{i}$ are generalized eigenvectors of the pencil $(\Gpos,\Gpr)$:
\begin{equation}\label{eq:wtilde}
\Gpos\, \widetilde{w}_{i} = \frac{1}{1+\delta_i^2}\,\Gpr\, \widetilde{w}_{i}.
\end{equation} 
    \end{itemize}       
\end{corollary}
Note that the definition of the class ${\Mc}_r^{-1}$ is analogous to that of ${\Mc}_r$.  Indeed, Lemma \ref{lemma_low_rank} in Appendix~\ref{sec:proofs} defines a bijection between these two classes.

\medskip

The vectors $\widetilde{w}_i = \Gpr^{-1} \widehat{w}_i$ are orthogonal with respect to the inner product defined by $\Gpr$. 
 By \eqref{eq:wtilde}, we also know that $\widetilde{w}_i$ minimizes the generalized Rayleigh quotient,
\begin{equation} \label{eq:Rayleigh_quotient_tilde}
  {\Rc}(z)=\frac{z^\top \Gpos z }{z^\top \Gpr \,z }=\frac{\Var(z^\top x\mid y)}{\Var(z^\top x)}, 
\end{equation}
over subspaces of the form $\widetilde{\mathcal{W}}_i={\rm span}^\perp (\widetilde{w}_j)_{j<i}$. This Rayleigh quotient is precisely the \textit{ratio of posterior to prior variance} along a particular direction, $z$, in the parameter space.  The smallest values that $\Rc$ can take over the subspaces $\widetilde{\mathcal{W}}_i$ are exactly the smallest generalized eigenvalues of $(\Gpos,\Gpr)$. In particular, the data are most informative along the first $r$ eigenvectors $\widetilde{w}_{i}$ and, since 
\begin{equation}\label{eq:optratio}
{\Rc}(\widetilde{w}_i)=
\frac{\Var(\widetilde{w}_i^\top x\mid y)}{\Var(\widetilde{w}_i^\top x)} = \frac{1}{1+\delta_i^2},
\end{equation}
the posterior variance is smaller than the prior variance by a factor of $(1+\delta_i^2)^{-1}$. In the span of the other eigenvectors, $(\widetilde{w}_{i}) _{i> r}$, the data are not as informative. Hence, $(\widetilde{w}_i)$ are the directions along which the ratio of posterior to prior variance is minimized. 
Furthermore, a simple computation shows that these directions also maximize the relative difference 
between prior and posterior variance normalized by the prior variance. Indeed, if the directions 
$(\widetilde{w}_i)$ minimize \eqref{eq:Rayleigh_quotient_tilde} then they must also maximize $1- {\Rc}(z)$, leading to:
\begin{equation}\label{eq:optratio_relvar}
1-{\Rc}(\widetilde{w}_i)=
\frac{\Var(\widetilde{w}_i^\top x)  -  \Var(\widetilde{w}_i^\top x\mid y)}{\Var(\widetilde{w}_i^\top x)} = \frac{\delta_i^2}{1+\delta_i^2}.
\end{equation}

\subsection{Optimal projector} 
Since the data are most informative on a subspace of the parameter
space, it should be possible to reduce the \textit{effective dimension} of the
inference problem in a manner that is consistent with the posterior
approximation. This is essentially the content of the following
corollary, which follows by a simple computation.
\medskip
\begin{corollary}[Optimal projector] \label{Cor:optimal_projector}
\highlight{
Let $\Gposh$ and the vectors $(\widehat{w}_i,\widetilde{w}_i)$ be defined as in Theorems \ref{Thm:main_th} and \ref{cor:Minv}.  Consider the reduced forward operator 
$\widehat{G}_r = G\circ P_r$, where $P_r$ is the oblique projector (i.e., $P_r^2 = P_r$ but $P_t^\top\neq P_r$):
 \begin{equation}\label{eq:projector}
P_r = \sum_{i=1}^r \widehat{w}_i \widetilde{w}_i^\top.
\end{equation}
}
 Then $\Gposh$ is  precisely the posterior covariance matrix corresponding to the Bayesian linear model: 
\begin{equation} \label{eq:reduced_Gauss}
y\mid x \sim \Gauss(\widehat{G}_r\, x,\,\Gobs),\quad x\sim \Gauss(0,\,\Gpr).
\end{equation}
\end{corollary}

The projected Gaussian linear model \eqref{eq:reduced_Gauss} reveals the intrinsic dimensionality of the inference problem.
The introduction of the optimal projector \eqref{eq:projector} is also useful in the context of dimensionality reduction for nonlinear inverse problems. In this case a particularly simple and effective approximation of the posterior density, $\pi_{\rm pos}(x \vert y)$, is of the form $\widehat{\pi}_{\rm pos}(x \vert y) \propto \pi(y; P_r \,x) \,\pi_{\rm pr}(x)$,
where $\pi_{\rm pr}$ is the prior density and $\pi(y;P_r\,x)$ is the density corresponding to the likelihood function with parameters constrained by the projector. 
The range of the projector can be determined by combining locally optimal data-informed subspaces from high-density regions in the support of the posterior distribution. This approximation is the subject of a related paper \cite{cui2014likelihood}. 

Returning to the linear inverse problem, notice also that the posterior mean of the projected model \eqref{eq:reduced_Gauss} might be used as an efficient approximation of the exact posterior mean. We will show in Section \ref{sec:mean} that this posterior mean approximation in fact minimizes the Bayes risk for a weighted squared-error loss among all low-rank linear functions of the data.

\subsection{Comparison with optimality in Frobenius norm} 

Thus far our optimality results for the approximation of $\Gpos$ have been restricted to the class of loss functions $\Lc$ given in Definition~\ref{def:lossfunc}. However, it is also interesting to investigate optimality in the metric defined by the Frobenius norm. Given any two matrices $A$ and $B$ of the same size, the Frobenius distance between them is defined as $\left\|A-B\right\|$, where $\left\|\,\cdot\right\|$ is the Frobenius norm. Note that the Frobenius distance does not exploit the structure of the positive definite cone of symmetric matrices. The matrix $\Gposh\in{\Mc}_r$ that minimizes the Frobenius distance from the exact posterior covariance matrix is given by: 
  \begin{equation} \label{minimizer_Frobenius}
     \Gposh=\Gpr-KK^\top, \quad KK^\top=\sum_{i=1}^r \lambda_i\, u_i u_i^\top, 
  \end{equation}
where $(u_i)$ are the directions corresponding to the $r$ largest eigenvalues of $\Gpr-\Gpos$. This result can be very different from the optimal approximation given in Theorem \ref{Thm:main_th}. In particular, the directions  $(u_i)$
are solutions of the eigenvalue problem
\begin{equation}\label{eq:vardiffeig}
\Gpr\, G^\top\Gy^{-1} G \,\Gpr\, u = \lambda u,
\end{equation} 
which maximize
\begin{equation} \label{eq:diff_var}
u^\top(\Gpr-\Gpos) u = \Var(u^\top x)-\Var(u^\top x\mid y).
\end{equation}
That is, while optimality in the F\"{o}rstner metric identifies directions that maximize the \textit{relative} difference between prior and posterior variance, the Frobenius distance favors directions that maximize only the absolute value of this difference.  There are many reasons to prefer the former. 
For instance, data might be informative along directions of low prior variance (perhaps due to inadequacies in prior modeling); a covariance matrix approximation that is optimal in Frobenius distance may ignore updates in these directions entirely. Also, if parameters of interest (i.e., components of $x$) have differing units of measurement, relative variance reduction provides a unit-independent way of judging the quality of a posterior approximation; this notion follows naturally from the second invariance property of $\df$ in \eqref{eq:dfprop}. 
From a computational perspective, solving the eigenvalue problem \eqref{eq:vardiffeig} is quite expensive compared to finding
\highlight{
 the generalized eigenpairs of the pencil $(H,\Gpr^{-1})$.} Finally, optimality in the Frobenius distance for an approximation of $\Gpos$ does not yield an optimality statement for the corresponding approximation of the posterior distribution, as shown in Lemma \ref{lemma:equivalence} for loss functions in $\Lc$.

\subsection{Suboptimal posterior covariance approximations} 
\label{s:suboptcov}

\subsubsection{Hessian-based and prior-based reduction schemes}
The posterior approximation described by Theorem~\ref{Thm:main_th} uses both Hessian and prior information. It is instructive to 
consider approximations of the linear Bayesian inverse problem that rely only on one or the other. As we will illustrate numerically 
in Section~\ref{sec:syntheticexamples}, these approximations can be viewed as natural limiting cases of our approach. They are also 
closely related to previous efforts in dimensionality reduction that propose only Hessian-based~\cite{lieberman2013hessian} or 
prior-based~\cite{marzouk2009dimensionality} reductions. In contrast with these previous efforts, here we will consider versions of 
Hessian- and prior-based reductions that do not discard prior information in the remaining directions. In other words, we will 
discuss posterior covariance approximations that remain in the form of \eqref{eq:class_M}---i.e., updating the prior covariance 
only in $r$ directions.

A Hessian-based reduction scheme updates $\Gpr$ in directions where the data have greatest influence in an absolute sense (i.e., not relative to the prior). This involves approximating the negative log-likelihood Hessian \eqref{eq:hessian} with a low-rank decomposition as follows: let $(s_i^2,v_i)$ be the eigenvalue-eigenvector pairs of $H$ with the ordering $s_i^2 \ge s_{i+1}^2$. Then a best low-rank approximation of $H$ in the Frobenius norm is given by:
\[
H\approx \sum_{i=1}^r s_i^2\, v_i v_i^\top =V_r S_r V_r^\top,
\]
where $v_i$ is the $i$th column of $V_r$ and $S_r=\mathrm{diag}\{ s_1^2 ,\ldots, s_r^2\}$. Using Woodbury's identity we then obtain an approximation of $\Gpos$ as a low-rank negative semidefinite update of $\Gpr$:
\begin{equation} \label{eq:Hessian_redu}
    \Gpos    \approx \left( V_r S_r V_r^\top + \Gpr^{-1} \right)^{-1} = \Gpr - \Gpr V_r \left( S_r^{-1} + V_r^\top \Gpr V_r \right)^{-1} V_r^\top \Gpr.
\end{equation}
This approximation of the posterior covariance matrix belongs to the class ${\Mc}_r$. Thus, Hessian-based reductions are in general 
suboptimal when compared to the optimal approximations defined in Theorem \ref{Thm:main_th}. Note that an equivalent way to obtain 
\eqref{eq:Hessian_redu} is to use a reduced forward operator of the form $\widehat{G}=G\circ V_r V_r^\top$, which is the composition 
of the original forward operator with a projector onto the leading eigenspace of $H$. In general, the projector $P_r=V_r V_r^\top$ 
is different from the optimal projector defined in Corollary \ref{Cor:optimal_projector} 
and is thus suboptimal. 

To achieve prior-based reductions, on the other hand, we restrict the Bayesian inference problem to directions in the parameter space that explain most of the prior variance. More precisely, we look for a rank-$r$ orthogonal projector, $P_r$, that minimizes the mean squared-error defined as:
\begin{equation} \label{eq:Prior_redu}
  \mathcal{E}\left( P_r \right) =  \Ex \left( \| x - P_r x \|^2 \right),
\end{equation}
where the expectation is taken over the prior distribution (assumed to have zero mean) and $\|\cdot\|$ is the standard Euclidean 
norm \cite{hua1998generalized}. Let $(t_i^2,u_i)$ be the eigenvalue-eigenvector pairs of $\Gpr$ ordered as $t_i^2 \ge t_{i+1}^2$. 
Then a minimizer of \eqref{eq:Prior_redu} is given by a projector, $P_r$,  onto the leading eigenspace of $\Gpr$ defined as:\,
\(
 P_r=\sum_{i=1}^r u_i u_i^\top = U_r U_r^\top,
\)
where $u_i$ is the $i$th column of $U_r$. The actual approximation of the linear inverse problem consists of using the projected forward operator, $\widehat{G}=G\circ U_r U_r^\top$.  By direct comparison with the optimal projector defined in Corollary \ref{Cor:optimal_projector}, 
we see that prior-based reductions are suboptimal in general. Also in this case, the posterior covariance matrix with the projected Gaussian model can be written as a negative semidefinite update of $\Gpr$:
\[
 \Gpos \approx \Gpr- U_r T_r [\,  (\, U_r^\top H U_r \,)^{-1} + T_r\,]^{-1} T_r U_r^\top,
\]
where $T_r=\mathrm{diag}\{ t_1^2 ,\ldots, t_r^2 \}$. The double matrix inversion makes this low-rank update computationally challenging to implement. It is also not optimal, as shown in Theorem \ref{Thm:main_th}.

To summarize, the Hessian and prior-based dimensionality reduction techniques are both suboptimal.  These methods do not take into account 
the interactions between the dominant directions of $H$ and those of $\Gpr$, nor the relative importance of these quantities. 
Accounting for such interaction is a 
key feature of the optimal covariance approximation described in Theorem \ref{Thm:main_th}. Section~\ref{sec:syntheticexamples} will 
illustrate conditions under which these interactions become essential.

\subsubsection{Connections with the BFGS Kalman filter}
\highlight{The linear Bayesian inverse problem analyzed in this paper can be interpreted as the analysis step of a linear Bayesian 
filtering problem \cite{evensen2007data}.  If the prior distribution corresponds to the forecast distribution at some time $t$, the 
posterior coincides with the so-called analysis distribution. In the linear case, with Gaussian process noise and observational errors, 
both of these distributions are Gaussian. The Kalman filter is a Bayesian solution to this filtering problem \cite{kalman1960new}. 
In \cite{auvinen2009large} the authors propose a computationally feasible way to implement (and approximate) this solution in large-scale 
systems. The key observation is that when solving an SPD linear system of the form $Ax=b$ by means of BFGS or limited memory BFGS 
(L-BFGS \cite{liu1989limited}), one typically obtains an approximation of $A^{-1}$ for free.  This approximation can be written as a 
low-rank correction of an arbitrary positive definite initial approximation matrix $A_{0}^{-1}$. The matrix $A_{0}^{-1}$ can be, for 
instance, the scaled identity. Notice that the approximation of $A^{-1}$ given by L-BFGS is full rank and positive definite. This 
approximation is in principle convergent as the storage limit of L-BFGS increases \cite{nocedal1980updating}. An L-BFGS approximation of 
$A$ is also possible \cite{wright1999numerical}.}

\highlight{There are many ways to exploit this property of the L-BFGS method. For example, in \cite{auvinen2009large} the posterior 
covariance is written as a low-rank update of the prior covariance matrix: $\Gpos=\Gpr-\Gpr G^{\top}\Gy^{-1}G\Gpr$, where 
$\Gy=\Gobs+G\Gpr G^{\top}$, and $\Gy^{-1}$ itself is approximated using the L-BFGS method. Since this approximation of $\Gy$ 
is full rank, however, this approach does not exploit potential low-dimensional structure of the inverse problem.
Alternatively, one can obtain an L-BFGS approximation of $\Gpos$ when solving the linear system $\Gpos^{-1} \, x=G^{\top}\Gobs^{-1}y$ for the posterior mean $\mupos$ \cite{auvinen2010variational}.  If one uses the prior covariance matrix as an initial approximation matrix, $A_{0}^{-1}$, then the resulting L-BFGS approximation of $\Gpos$ can be written as a low-rank update of $\Gpr$.  This approximation format is similar to the one discussed in \cite{flath2011fast} and advocated in this paper.  However, the approach of \cite{auvinen2010variational} (or its ensemble version \cite{solonen2012variational}) does not correspond to any known optimal approximation of the posterior covariance matrix, nor does it lead to any optimality statement between the corresponding probability distributions. This is an important contrast with the present approach, which we will revisit numerically in Section~\ref{sec:syntheticexamples}. }

\section{Optimal approximation of the posterior mean} \label{sec:mean}

In this section, we develop and characterize fast approximations of the posterior mean that can be used, for instance, to accelerate repeated inversion with multiple data sets.
\highlight{Note that we are not proposing alternatives to the efficient computation of the posterior mean for a
  single realization of the data. This task can already be accomplished with current state-of-the-art iterative solvers
  for regularized least-squares problems
  \cite{hestenes1952methods,paige1982algorithm,akccelik2006parallel,meng2014lsrn,barrett1994templates}. Instead, we are
  interested in constructing \emph{statistically optimal approximations}\footnote{
	\highlightTwo{  
  We will precisely define this notion of optimality in Section~\ref{sec:meanoptresults}.}}
  of the posterior mean as linear functions of
  the data. 
  \highlightTwo{
  That is, we seek a matrix $A$, from an approximation class to be defined shortly, such that the posterior
  mean can be approximated as $\mupos \approx A \,y$.
  We will investigate different approximation classes for $A$;
  in particular, we will only consider approximation classes for which applying $A$ to a vector $y$ is relatively inexpensive.
  Computing such a matrix $A$ is more expensive than solving a single linear system associated with the posterior
  precision matrix to determine the posterior mean.
  However, once $A$ is computed, it can be applied inexpensively to
  any realization of the data.\footnote{\highlightTwo{In particular, applying $A$ is much cheaper than solving a linear system.}}}  
  Our approach is therefore justified when the posterior mean must be evaluated for multiple instances of the
  data. This approach can thus be viewed as an offline--online strategy, where a more costly but data-independent
  offline calculation is followed by fast online evaluations.
  \highlightTwo{
  Moreover, we will show that these approximations can be obtained from an optimal approximation of the
  posterior covariance matrix (cf.\ Theorem \ref{Thm:main_th}) with minimal additional cost.
  Hence, if one is interested in \textit{both} the  posterior mean and
  covariance matrix (as is often the case in the Bayesian approach
  to inverse problems), then the approximation formulas we propose can
  be more efficient than standard approaches %
  even for a single realization of the data.
  }
  }

\subsection{Optimality results}
\label{sec:meanoptresults}
For the Bayesian linear model defined in \eqref{eq:GBLM}, the posterior mode is equal to the posterior mean, $\mupos=\Ex(x\vert y)$, 
which is in turn the minimizer of the Bayes risk for squared-error loss \cite{lehmann1998theory,lindley1972bayes}. We first review this 
fact and establish some basic notation. Let $S$ be an SPD matrix and let
\[
L(\delta(y),x) = \left ( x-\delta(y) \right)^\top  \! S  \left (x-\delta(y) \right )=\left \| x-\delta(y) \right \|_S^2
\]
be the loss incurred by the estimator $\delta(y)$ of $x$. The Bayes risk, $R \left (\delta(y),x \right )$, of $\delta(y)$ is defined as the average loss over the joint distribution of $x$ and $y$ \cite{carlin2009bayesian,lehmann1998theory}:
$R(\delta(y),x)=\Ex \left (\, L(\delta(y),x) \, \right )$. Since 
\begin{equation}\label{eq:riskdec}
R(\delta(y),x) = \Ex \left  ( \|\,\delta(y) - \mupos\,\|_S^2 \right ) + \Ex \left ( \|\mupos - x\|_S^2 \right ),
\end{equation}
it follows that $\delta(y)=\mupos$ minimizes the Bayes risk over all estimators of $x$.

To study approximations of $\mupos$, we use the squared-error loss function defined by the Mahalanobis distance \cite{christensen1987plane} induced by $\Gpos^{-1}$:
\(
L(\delta(y),x)= \left  \| \delta(y)-x \right \|_{\Gpos^{-1}}^2 .
\)
This loss function accounts for the geometry induced by the posterior measure on the parameter space, penalizing errors in the approximation of $\mupos$ more strongly in directions of lower posterior variance.

Under the assumption of zero prior mean, $\mupos$ is a linear function of the data. Hence we seek approximations of $\mupos$ of the form $Ay$, where $A$ is a matrix in a class to be defined. Our goal is to obtain fast posterior mean approximations that can be applied repeatedly to multiple realizations of $y$. We consider two classes of approximation matrices: 
\begin{equation}\label{eq:Aclass}
\Ac_r := \left\{ A\,:\,{\rm rank}(A)\le r \right\} \quad \text{and}
\quad \widehat{\Ac}_r := \left\{ A=( \Gpr-B )\,G^\top\Gobs^{-1}\,:\,
{\rm rank}(B)\le r \right\}.
\end{equation}
The class $\Ac_r$ consists of low-rank matrices; it is standard in the statistics literature \cite{hua1998generalized}. The class $\widehat{\Ac}_r$, on the other hand, can be understood via comparison with \eqref{eq:postmoms}; it simply replaces $\Gpos$ with a low-rank negative semidefinite update of $\Gpr$. We shall henceforth use $\Ac$ to denote either of the two classes above. 

Let $R_{\Ac}(Ay,x)$ be the Bayes risk of $Ay$ subject to $A\in \Ac$. We may now restate our goal as: find a matrix, $A^\ast \in \Ac$, 
that minimizes the Bayes risk $R_{\Ac}(Ay,x)$. That is, find $A^\ast \in \Ac$ such that  
\begin{equation} \label{eq:form_Bayes_risk}
R_{\Ac}(A^\ast y,x) = \min_{A\in\Ac}\Ex(\,\left\|Ay - x\right\|^2_{\Gpos^{-1}}\,).
\end{equation}
The following two theorems show that for either class of approximation matrices, $\Ac_r$ or $\widehat{\Ac}_r$, this problem admits a particularly simple analytical solution that exploits the structure of the optimal approximation of $\Gpos$. The proofs of the theorems rely 
\highlightTwo{
on a result developed independently by Sondermann \cite{sondermann1986best} and 
Friedland  \& Torokhti \cite{friedland2007generalized}, and are given in Appendix~\ref{sec:proofs}.
}
We also use the fact that $\Ex \left ( \left \| \mupos - x \right \|_{\Gpos^{-1}}^2 \right )=\ell$, where $\ell$ is the dimension of the parameter space.
\medskip

\begin{theorem} \label{thm:mean_approx_lowrank}
\highlight{
Let $(\delta_i^2,\widehat{w}_i)$ be defined as in Theorem \ref{Thm:main_th} and let $(\widehat{v}_i)$ be generalized eigenvectors of the pencil $(G\Gpr G^\top , \Gobs)$ associated with a non-increasing sequence of eigenvalues, 
 \highlightTwo{ with the normalization $\widehat{v}_i^\top \, \Gobs \, \widehat{v}_i  = 1$}. Then:  }
 \begin{itemize}
\item[(i)] A solution of \eqref{eq:form_Bayes_risk} for $A\in \Ac_r$ is given by:
 \begin{equation} \label{eq: mean_approx_lowrank}
     A^\ast =  \sum_{i= 1}^r  \frac{\delta_i}{1+\delta_i^2}\, \widehat{w}_i \widehat{v}_i^\top,
 \end{equation}
\item[(ii)] The corresponding minimum Bayes risk over $\Ac_r$ is given by:
 \begin{equation}\label{eq:RAr}
 R_{\Ac_r}(A^\ast y,x)= \Ex \left ( \left \|A^\ast y - \mupos \right \|^2_{\Gpos^{-1}} \right ) 
+ \Ex \left ( \left \|\mupos - x \right \|_{\Gpos^{-1}}^2 \right ) = \sum_{i>r} \delta_i^2 + \ell.
\end{equation}
\end{itemize}
\end{theorem}
Notice that the rank-$r$ posterior mean approximation given by Theorem \ref{thm:mean_approx_lowrank} coincides with the posterior mean of the projected linear Gaussian model defined in \eqref{eq:reduced_Gauss}.  Thus, applying this approximation to a new realization of the data requires only a \textit{low-rank} matrix-vector product, a computationally trivial task.
\highlightTwo{ We define the quality of a posterior mean approximation as the minimum Bayes risk \eqref{eq:RAr}.
Notice, however, that for a given rank $r$ of the approximation, \eqref{eq:RAr} depends on the eigenvalues that have not yet been computed. Since $(\delta_i^2)$ are determined in order of decreasing magnitude, \eqref{eq:RAr} can be easily bounded (cf.\ discussion after Theorem \ref{Thm:main_th}). The forthcoming minimum Bayes risk \eqref{eq:RArh} can be bounded analogously.}
\vspace{5pt}
\begin{remark}{\rm
\highlight{
Equation \eqref{eq: mean_approx_lowrank} can be interpreted as the truncated GSVD solution of a Tikhonov regularized
linear inverse problem \cite{hansen1989regularization} (with unit regularization parameter). Hence, Theorem \ref{thm:mean_approx_lowrank} also describes a Bayesian property of the (frequentist) truncated GSVD estimator.}
}
\end{remark}
\vspace{5pt}
\begin{remark}{\rm 
\highlight{
If factorizations of the form $\Gpr = \Gprs \Gprs^\top$ and $\Gobs = \Gobss \Gobss^\top$ are readily available, then we can characterize the triplets $(\delta_i, \widehat{w}_i, \widehat{v}_i)$  from a singular value decomposition, 
$\Gobss^{-1} G \Gprs = \sum_{i\ge 1} \delta_i v_i w_i^\top$,
of the matrix $\Gobss^{-1} G \Gprs$ with the transformations 
$\widehat{w}_i = \Gprs w_i$, $\widehat{v}_i = \Gobss^{-\top} v_i$ and the ordering $\delta_i \ge \delta_{i+1}$.  
In particular, the approximate posterior mean can be written as:
\begin{equation} \label{eq: mean_lowrank_pseudoinverse}
	\muposr = \Gprs ( \Gobss^{-1} G \Gprs )_r^{\rm Tikh } \Gobss^{-1} y
\end{equation}   
where $ ( \Gobss^{-1} G \Gprs )_r^{\rm Tikh}$ is the best rank-$r$ approximation to a Tikhonov regularized inverse.\footnote{With unit regularization parameter and identity regularization operator \cite{hansen1998rank}.} 
That is, for any matrix $A$, $(A)_r$ is the best rank-$r$ approximation of $A$ (e.g., computed via SVD), whereas $(A)^{\rm Tikh}:=(A^\top A + I )^{-1}A^\top$. }
}
\end{remark}
\medskip

\begin{theorem}  \label{thm:mean_approx_fullrank}
Let $\Gposh\in {\Mc}_r$ be the optimal approximation of $\Gpos$ defined in 
Theorem \ref{Thm:main_th}. Then:
\begin{itemize}
\item[(i)] A  solution of \eqref{eq:form_Bayes_risk} for $A\in \widehat{\Ac}_r$ is given by:
 \begin{equation} \label{eq: mean_approx_fullrank}
     \widehat{A}^\ast =\Gposh\, G^\top \Gobs^{-1} .
 \end{equation}
\item[(ii)] The corresponding minimum Bayes risk over $\widehat{\Ac}_r$ is given by:
\begin{equation}\label{eq:RArh}
  R_{\widehat{\Ac}_r}(\widehat{A}^\ast y,x) = 
\Ex \left ( \left \|\widehat{A}^\ast y - \mupos \right \|^2_{\Gpos^{-1}} \right ) + \Ex \left ( \left \|\mupos - x \right \|_{\Gpos^{-1}}^2 \right )
= \sum_{i>r} \delta_i^6 + \ell.
\end{equation}
\end{itemize}
\end{theorem}

Once the optimal approximation of $\Gpos$ described in Theorem~\ref{thm:mean_approx_fullrank} is computed, the cost of 
approximating $\mupos$ for a new realization of $y$ is dominated by the adjoint and prior solves needed to apply $G^\top$ and $\Gpr$, 
respectively. Combining the optimal approximations of $\mupos$ and $\Gpos$ given by Theorems \ref{thm:mean_approx_fullrank} and 
\ref{Thm:main_th}, respectively, yields a complete approximation of the Gaussian posterior distribution. This is precisely the 
approximation adopted by the stochastic Newton MCMC method \cite{martin2012stochastic} to describe the Gaussian proposal distribution 
obtained from a local linearization of the forward operator of a nonlinear Bayesian inverse problem. Our results support the 
algorithmic choice of \cite{martin2012stochastic} with precise optimality statements.

It is worth noting that the two optimal Bayes risks, \eqref{eq:RAr} and \eqref{eq:RArh}, depend on the parameter, $r$, that defines the dimension of the corresponding approximation classes $\Ac_r$ and $\widehat{\Ac}_r$. In the former case, $r$ is the rank of the optimal matrix that defines the approximation. In the latter case, $r$ is the rank of a negative update of $\Gpr$ that yields the posterior covariance matrix approximation. We shall thus refer to the estimator given by Theorem \ref{thm:mean_approx_lowrank} as the low-rank approximation and to the estimator given by Theorem \ref{thm:mean_approx_fullrank} as the low-rank \textit{update} approximation. In both cases, we shall refer to $r$ as the order of the approximation.  A posterior mean approximation of order $r$ will be called \textit{under-resolved} if more than $r$ 
\highlight{
generalized eigenvalues of the pencil $(H,\Gpr^{-1})$  are greater than one. If this is the case, then using the low-rank update approximation is not appropriate because the associated Bayes risk includes high-order powers of eigenvalues of $(H,\Gpr^{-1})$  that are greater than one. Thus, under-resolved approximations tend to be more accurate when using the low-rank approximation. As we will show in Section~\ref{s:examples}, this estimator is also less expensive to compute than its counterpart in Theorem \ref{thm:mean_approx_fullrank}. If, on the other hand, fewer than $r$ eigenvalues of $(H,\Gpr^{-1})$  are greater than one, then the optimal low-rank \textit{update} estimator will have better performance than the optimal low-rank estimator in the following sense:
}
 \[
 0<R_{\Ac_r}(A^\ast y,x) -   R_{\widehat{\Ac}_r}(\widehat{A}^\ast y,x)   = \sum_{i>r} \delta_i^2 \left( 1+\delta_i^2 \right) \left( 1-\delta_i^2 \right).
 \]

\subsection{Connection with ``priorconditioners''}
\highlight{
In this subsection we draw connections between the low-rank approximation of the posterior mean given in Theorem
\ref{thm:mean_approx_lowrank} and the regularized solution of a discrete ill-posed inverse problem, $y=Gx+\vare$ (using
the notation of this paper), as presented in \cite{calvetti2005priorconditioners,calvetti2007preconditioned}. In
\cite{calvetti2005priorconditioners,calvetti2007preconditioned}, the authors propose an early stopping regularization
using iterative solvers preconditioned by prior statistical information on the parameter of interest, say
$x\sim\mathcal{N}(0,\Gpr)$, and on the noise, say $\vare \sim \Nc(0,\Gobs)$.\footnote{It suffices to consider a
  Gaussian approximation to the distribution of $x$ and $\vare$} That is, if factorizations $\Gpr = \Gprs \Gprs^\top$
and $\Gobs = \Gobss \Gobss^\top$ are available, then \cite{calvetti2005priorconditioners} provides a solution, $x=\Gprs
\, q$, to the inverse problem, where $q$ comes from an early stopping regularization applied to the preconditioned linear
system: 
\begin{equation}
 \Gobss^{-1} G \Gprs q = \Gobss^{-1} y. 
\end{equation}
The iterative method of choice in this case is the CGLS algorithm \cite{calvetti2005priorconditioners,
  hanke1995conjugate} (or GMRES for nonsymmetric square systems \cite{calvetti2002regularizing}) equipped with a proper stopping criterion (e.g., the discrepancy principle \cite{kaipio2005statistical}).  Although the approach of
\cite{calvetti2005priorconditioners} is not exactly Bayesian, 
we can still use the optimality results of Theorem \ref{thm:mean_approx_lowrank} to justify the observed good performance of this
particular form of regularization.  By a property of the CGLS algorithm, the $r$th iterate, $x^r=\Gprs q^r$, satisfies:
\begin{equation} \label{eq:CGLS_property}
q^r =  \operatorname*{argmin}_{q\in \mathcal{K}_r \left( \widehat{H} , \widehat{y}  \right) } \Vert \, \Gobss^{-1}y -  \Gobss^{-1} G
\Gprs  q  \, \Vert .
\end{equation}
where $\mathcal{K}_r ( \widehat{H} , \widehat{y} ) $ is the $r$--dimensional Krylov subspace associated with the matrix $\widehat{H}=\Gprs^\top H \Gprs$ and starting vector $\widehat{y}= \Gprs^\top G^\top \Gobs^{-1}y$.  It was shown in \cite{hestenes1975pseudoinversus} that the CGLS solution, at convergence, can be written as $x^*=\Gprs(\Gobss^{-1}G\Gprs)^{\dagger}\Gobss^{-1}y$, where $(\,\cdot \,)^{\dagger}$ denotes the Moore-Penrose pseudoinverse \cite{Moore1920,penrose1955generalized}.  
To highlight the differences between the CGLS solution and \eqref{eq: mean_lowrank_pseudoinverse}, 
we assume that $ \mathcal{K}_{r}(\widehat{H},y) \approx {\text{ran}}(W_{r})$ for all $r$, 
where $W_{r}= \left [ w_{1} \, |\cdots | \, w_{r} \right ]$, 
\highlightTwo{ ${\text{ran}}(A)$ denotes the range of a matrix $A$}, 
and $\widehat{H}=\sum_{i}\delta_{i}^{2}w_{i}w_{i}^{\top}$ is an SVD of $\widehat{H}$.  
Notice that the condition $\mathcal{K}_{r}(\widehat{H},y) \approx {\rm ran}(W_{r})$ is usually quite reasonable for moderate values of $r$.
This practical observation is at the heart of the Lanczos iteration for symmetric eigenvalue problems \cite{lanczos1950iteration}.
With simple algebraic manipulations we conclude that:
\begin{equation} \label{eq: sol_priorconditioner}
x^{r} \approx \Gprs(\Gobss^{-1}G\Gprs)_{r}^{\dagger}\,\Gobss^{-1}y .
\end{equation}
Recall from \eqref{eq: mean_lowrank_pseudoinverse} that the optimal rank--$r$ approximation of the posterior mean defined in Theorem \ref{thm:mean_approx_lowrank}
can be written as:
\begin{equation} \label{eq: mean_lowrank_pseudoinverse_again}
	\muposr = \Gprs ( \Gobss^{-1} G \Gprs )_r^{ \rm Tikh  } \Gobss^{-1} y .
\end{equation}   
The only difference between \eqref{eq: sol_priorconditioner} and \eqref{eq: mean_lowrank_pseudoinverse_again}
is the use of a Tikhonov-regularized inverse in \eqref{eq: mean_lowrank_pseudoinverse_again} as opposed to 
a Moore-Penrose pseudoinverse. If $\Gobss^{-1} G \Gprs = \sum_{i\ge 1} \delta_i v_i w_i^\top$ is a reduced SVD of the matrix 
$\Gobss^{-1} G \Gprs$, then:
\begin{equation}
  (\Gobss^{-1}G\Gprs)_{r}^{\dagger}=\sum_{i\le r} \frac{1}{\delta_i} w_i v_i^\top, \qquad 
   ( \Gobss^{-1} G \Gprs )_r^{ \rm Tikh  } = \sum_{i\le r} \frac{\delta_i}{1+\delta_i^2} w_i v_i^\top.
\end{equation}
These two matrices are {\it nearly} identical for values of $r$ corresponding to $\delta_r^2$ 
greater than one\footnote{In Section \ref{s:examples} we show that by the time 
 we start including generalized eigenvalues $\delta_i^2 \approx 1$ 
in \eqref{eq: mean_approx_lowrank}, the approximation of the posterior mean  is usually already satisfactory.
Intuitively, this means that all the directions in parameter space where the data are more informative than the prior have
been considered.}  
(assuming the ordering $\delta_i^2 \ge \delta_{i+1}^2$).
Beyond this regime, it might be convenient to stop the CGLS solver to obtain \eqref{eq: sol_priorconditioner} 
(i.e., early stopping regularization).
The similarity of these expressions is quite remarkable since \eqref{eq: mean_lowrank_pseudoinverse_again}  
was derived as the minimizer of the optimization problem \eqref{eq:form_Bayes_risk} with $\Ac = \Ac_r$.
This informal argument may explain why {\it priorconditioners} perform so well in applications 
\cite{calvetti2012left, homa2013bayesian}. Yet we remark that the goals of Theorem \ref{thm:mean_approx_lowrank} and 
\cite{calvetti2005priorconditioners} are still quite different;
\cite{calvetti2005priorconditioners} is concerned with preconditioning techniques for early stopping regularization of 
ill-posed inverse problems, whereas Theorem \ref{thm:mean_approx_lowrank} is concerned with 
statistically optimal approximations of the posterior mean in the Bayesian framework.
}

\begin{algorithm}
\caption{Optimal \textit{low-rank} approximation of the posterior mean}
 INPUT: \highlight{
 forward and adjoint models $G$, $G^\top$; prior and noise precisions  $\Gpr^{-1}$, $\Gobs^{-1}$;  approximation order $r\in\mathbb{N}$
 }\\
 OUTPUT: approximate posterior mean $\mu^{(r)}_{\rm pos}(y)$
\begin{algorithmic}[1]
  \small
  \STATE   \highlight{ Find the $r$ leading generalized eigenvalue-eigenvector pairs $(\delta_i^2,\widehat{w}_i)$ of the pencil $(G^\top \Gobs^{-1} G , \Gpr^{-1})$
  \STATE Find the $r$ leading  generalized eigenvector pairs $(\widehat{v}_i)$ of the pencil $(G\,\Gpr G^\top , \Gobs)$
  \STATE For each new realization of the data $y$, compute 
    $\mu^{(r)}_{\rm pos}(y)= \sum_{i=1}^r \delta_i  (1+\delta_i^2)^{-1} \widehat{w}_i  \widehat{v}_i^\top   \,y$.}
\end{algorithmic}
\end{algorithm}

\begin{algorithm}
\caption{Optimal \textit{low-rank update} approximation of the posterior mean}
 INPUT: forward and adjoint models $G$, $G^\top$; prior and noise precisions  $\Gpr^{-1}$, $\Gobs^{-1}$;  approximation order $r\in\mathbb{N}$ \\
 OUTPUT: approximate posterior mean $\widehat{\mu}^{(r)}_{\rm pos}(y)$
\begin{algorithmic}[1]
  \small
  \STATE Obtain $\Gposh$ as described in Theorem~\ref{Thm:main_th}. %
  \STATE For each new realization of the data $y$, compute 
    $\widehat{\mu}^{(r)}_{\rm pos}(y)= \Gposh \, G^\top \Gobs^{-1}\, y$.
  
\end{algorithmic}
\end{algorithm}

\section{Numerical examples}
\label{s:examples}
Now we provide several numerical examples to illustrate the theory developed in the preceding sections. We start with a synthetic example to demonstrate various posterior covariance matrix approximations, and continue with two more realistic linear inverse problems where we also study posterior mean approximations.

\subsection{Example 1: Hessian and prior with controlled spectra} 
\label{sec:syntheticexamples}
We begin by investigating the approximation of $\Gpos$ as a negative semidefinite update of $\Gpr$. We compare the optimal approximation obtained in Theorem \ref{Thm:main_th} with the Hessian-, prior-, and BFGS-based reduction schemes discussed in Section~\ref{s:suboptcov}. The idea is to reveal differences between these approximations by exploring regimes where the data have differing impacts on the prior information. Since the directions defining the optimal update are the generalized eigenvectors of the pencil $(H,\Gpr^{-1})$, we shall also refer to this update as the \textit{generalized} approximation. 

To compare these approximation schemes, we start with a simple example with diagonal Hessian and prior covariance matrices: $G=I$, $\Gobs = \mathrm{diag}\{\sigma_i^2\}$, and $\Gpr = \mathrm{diag}\{\lambda_i^2\}$. Since the forward operator $G$ is the identity, this problem can (loosely) be thought of as denoising a signal $x$. In this case, $H=\Gobs^{-1}$ and $\Gpos = \mathrm{diag}\{\lambda_i^2\sigma_i^2/(\sigma_i^2+\lambda_i^2)\}$. The ratios of posterior to prior variance in the canonical directions $(e_i)$ are 
\[
\frac{\Var(e_i^\top x\mid y)}{\Var(e_i^\top x)} = \frac{1}{1+\lambda_i^2/\sigma_i^2}.
\]
We note that if the observation variances $\sigma_i^2$ are constant, $\sigma_i=\sigma$, then the directions of greatest
variance reduction are those corresponding to the largest prior variance. Hence the prior distribution alone determines the most informed directions, and the prior-based reduction is as effective as the generalized one. On the other hand, if the prior variances $\lambda_i^2$ are constant, $\lambda_i =\lambda$, but the $\sigma_i$ vary, then the directions of highest variance reduction are those corresponding to the smallest noise variance. This time the noise distribution
alone determines the most important directions, and Hessian-based reduction is as effective as the generalized one.
In the case of more general spectra, the important directions depend on the ratios $\lambda_i^2/\sigma_i^2$ and thus one has to use the information provided by both the noise and prior distributions. This is done naturally by the generalized reduction. 

We now generalize this simple case by moving to full matrices $H$ and $\Gpr$ with a variety of prescribed spectra. We assume that $H$ and $\Gpr$ have SVDs of the form $H =U\Lambda U^\top$ and $\Gpr = V\widetilde{\Lambda} V^\top$, where $\Lambda = \mathrm{diag}\{\lambda_1,\ldots,\lambda_n \}$ and $\widetilde{\Lambda} = \mathrm{diag}\{\widetilde{\lambda}_1,\ldots,\widetilde{\lambda}_n \}$ with
\[
\lambda_k = \lambda_0/k^{\alpha}+\tau\quad \mathrm{and}\quad \widetilde{\lambda}_k = \widetilde{\lambda}_0/k^{\widetilde{\alpha}}+
\widetilde{\tau}.
\]
To consider many different cases, the orthogonal matrices $U$ and $V$ are randomly and independently generated uniformly over the orthogonal group \cite{stewart1980}, leading to different realizations of $H$ and $\Gpr$. In particular,  $U$ and $V$ are computed with a $QR$ decomposition of a square matrix of independent standard Gaussian entries using a Gram-Schmidt orthogonalization. (In this case, the standard Householder reflections cannot be used.)

\highlight{
Before discussing the results of the first experiment, we explain our implementation of BFGS-based reduction.  We ran the BFGS optimizer with a dummy quadratic optimization target $\mathcal{J}(x) = \frac{1}{2} \, x^\top \Gpos^{-1} x$ and used $\Gpr$ as the initial approximation matrix for $\Gpos$. Thus, the BFGS approximation of the posterior covariance matrix can be written as $\Gpos = \Gpr - KK^\top$ for some rank--$r$ matrix $K$.  The rank--$r$ update is constructed by running the BFGS optimizer for $r$ steps from random initial conditions as shown in \cite{auvinen2010variational}.
Note that in order to obtain results for sufficiently high-rank updates, we use BFGS rather than L-BFGS in our numerical examples. While \cite{auvinen2009large, auvinen2010variational} in principle employ L-BFGS, the results in these papers use a number of optimization steps roughly equal to the number of vectors stored in L-BFGS; 
our approach thus is comparable to \cite{auvinen2009large, auvinen2010variational}.
Nonetheless, some results for the highest-rank BFGS updates are not plotted in Figures \ref{fig:randomized_priorfixed} and \ref{fig:randomized_hessfixed}, as the optimizer converged so close to the optimum that taking further steps resulted in numerical instabilities.}

Figure \ref{fig:randomized_priorfixed} summarizes the results of the first experiment. The top row shows the prescribed spectra of $H^{-1}$ (red) and $\Gpr$ (blue). The parameters describing the eigenvalues of $\Gpr$ are fixed to $\widetilde{\lambda}_0=1$, $\widetilde{\alpha}=2$, and $\widetilde{\tau}=10^{-6}$. The corresponding parameters for $H$ are given by $\lambda_0=500$ and $\tau=10^{-6}$ with $\alpha=0.345$ (left), $\alpha=0.690$ (middle), and $\alpha=1.724$ (right). Thus, moving from the leftmost column to the rightmost column, the data become increasingly less informative. The second row in the figure shows the F\"{o}rstner distance between $\Gpos$ and its approximation, $\Gposh=\Gpr-KK^\top$, as a function of the rank of $KK^\top$ for 100 different realizations of $H$ and $\Gpr$. The third row shows, for each realization of $(H,\Gpr)$ and for each fixed rank of $KK^\top$, the difference between the F\"{o}rstner distance obtained with a prior-, Hessian-, or BFGS-based dimensionality reduction technique and the minimum distance obtained with the generalized approximation. All of these differences are positive---a confirmation of Theorem \ref{Thm:main_th}. 
But Figure \ref{fig:randomized_priorfixed} also shows interesting patterns consistent with the observations made for the simple example above. When the spectrum of $H$ is basically flat (left column), the directions along which the prior variance is reduced the most are likely to be those corresponding to the largest prior variances, and thus a prior-based reduction is almost as effective as the generalized one (as seen in the bottom two rows on the left). 
As we move to the third column, eigenvalues of $H^{-1}$ increase more quickly. The data provide significant information only on a lower-dimensional subspace of the parameter space. In this case, it is crucial to combine this information with the directions in the parameter space along which the prior variance is the greatest. The generalized reduction technique successfully accomplishes this task, whereas the prior and Hessian reductions fail as they focus either on $\Gpr$ or $H$ alone; the key is to combine the two.
\highlight{The BFGS update performs remarkably well across all three configurations of the Hessian spectrum, although it is clearly suboptimal compared to the generalized reduction.}

In Figure \ref{fig:randomized_hessfixed} the situation is reversed and the results are symmetric to those of Figure \ref{fig:randomized_priorfixed}.  The spectrum of $H$ (red) is now kept fixed with parameters $\lambda_0=500$, $\alpha=1$, and $\tau=10^{-9}$, while the spectrum of $\Gpr$ (blue) has parameters $\widetilde{\lambda}_0=1$ and $\widetilde{\tau}=10^{-9}$ with decay rates increasing from left to right: $\widetilde{\alpha}=0.552$ (left), $\widetilde{\alpha}=1.103$ (middle), and $\widetilde{\alpha}=2.759$ (right). In the first column,  the spectrum of the prior is nearly flat. That is, the prior variance is almost equally spread along every direction in the parameter space. In this case, the eigenstructure of $H$ determines the  directions of greatest variance reduction, and the Hessian-based reduction is almost as effective as the generalized one. As we move towards the third column, the spectrum of $\Gpr$ decays more quickly. The prior variance is restricted to lower-dimensional subspaces of the parameter space. Mismatch between prior- and Hessian-dominated directions then leads to poor performance of both the prior- and Hessian-based reduction techniques.  However, the generalized reduction performs well also in this more challenging case.
\highlight{The BFGS reduction is again empirically quite effective in most of the configurations that we consider. It is not always better than the prior- or Hessian-based techniques when the update rank is low, or when the prior spectrum decays slowly; for example, Hessian-based reduction is more accurate than BFGS across all ranks in the first column of Figure \ref{fig:randomized_hessfixed}. But when either the prior covariance or the Hessian have quickly decaying spectra, the BFGS approach performs almost as well as the generalized reduction. Though this approach remains suboptimal, its approximation properties deserve further theoretical study.}

\begin{figure}[!htp]
\begin{center}
\includegraphics[width=.8\textwidth]{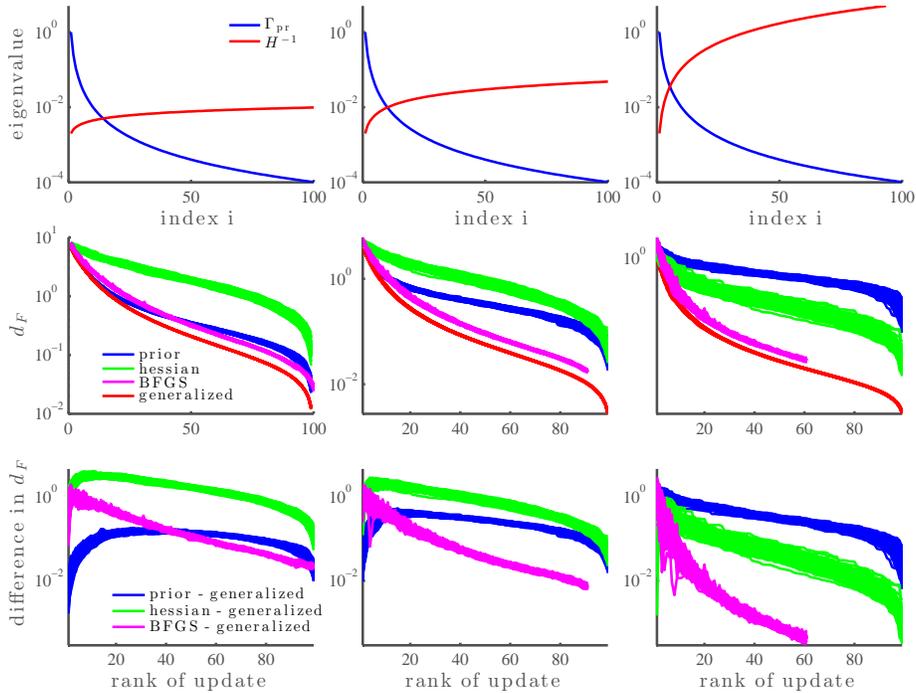}
\caption{Top row: Eigenspectra of $\Gpr$ (blue) and $H^{-1}$ (red) for three values for the decay rate of the 
eigenvalues of $H$: 
$\alpha=0.345$ (left), $\alpha=0.690$ (middle) and $\alpha=1.724$ (right). Second row: F\"{o}rstner distance between $\Gpos$
and its approximation versus the rank of the update for 100 realizations of
$\Gpr$ and $H$ using prior-based (blue), Hessian-based (green), BFGS-based (magenta) and optimal (red) updates. 
Bottom row: Differences of posterior covariance approximation error (measured with the F\"{o}rstner metric) between the prior-based and optimal updates (blue), between the Hessian-based and optimal updates (green), and between the BFGS-based and optimal updates (magenta).}
\label{fig:randomized_priorfixed}
\end{center}
\end{figure}

\begin{figure}[!htp]
\begin{center}
\includegraphics[width=.8\textwidth]{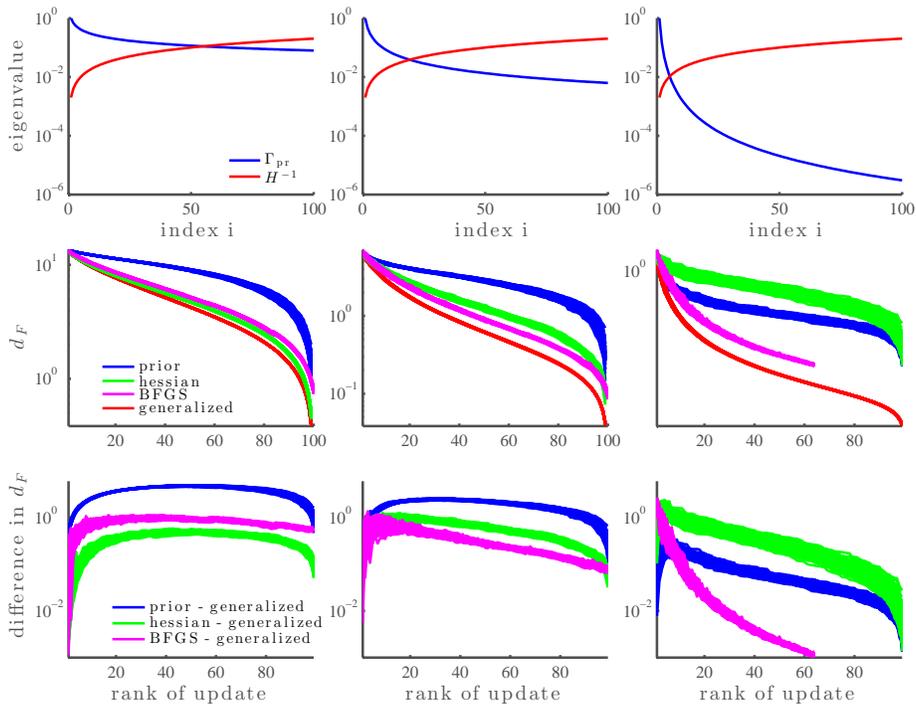}
\caption{Analogous to Figure \ref{fig:randomized_priorfixed} but this time the spectrum of $H$ is fixed, while that of $\Gpr$ has varying decay rates: $\widetilde{\alpha}=0.552$ (left), $\widetilde{\alpha}=1.103$ (middle) and $\widetilde{\alpha}=2.759$ (right).}
\label{fig:randomized_hessfixed}
\end{center}
\end{figure}

\subsection{Example 2: X-ray tomography}
\label{sec:tomo}
We consider a classical inverse problem of X-ray computed tomography (CT), where X-rays travel from sources to detectors through an object of interest. The intensities from multiple sources are measured at the detectors, the goal is to reconstruct the density of the object. In this framework, we investigate the performance of the optimal mean and covariance matrix approximations presented in Sections \ref{sec:theory} and \ref{sec:mean}.
\highlight{This synthetic example is motivated by a real application: real-time X-ray imaging of logs that enter a saw mill for the purpose of automatic quality control.  For instance, in the system commercialized by Bintec (www.bintec.fi), logs enter the X-ray system on fast-moving conveyer belt and fast reconstructions are needed. The imaging setting (e.g., X-ray source and detector locations) and the priors are fixed; only the data changes from one log cross-section to another. The basis for our posterior mean approximation can therefore be pre-computed, and repeated inversions can be carried out quickly with direct matrix formulas.}

We model the absorption of an X-ray along a line, $\ell_i$, using Beer's law: 
\begin{equation}
I_d = I_s\exp\left(-\int_{\ell_i}{\!\!x(s)ds} \right),
\label{eq:beer}
\end{equation}
where $I_d$ and $I_s$ are the intensities at the detector and at the source, respectively, and $x(s)$ is the density of the object at position $s$ on the line $\ell_i$. The computational domain is discretized into a grid and the density is 
assumed to be constant within each grid cell. The line integrals are approximated as
\begin{equation}
\int_{\ell_i}{\!\!x(s)ds}\, \approx \sum_{j=1}^{\mathrm{\#\ of\ cells}} \!\!\!g_{ij} x_j, %
\end{equation}
where $g_{ij}$ is the length of the intersection between line $\ell_i$ and cell $j$, and $x_j$ 
is the unknown density in cell $j$. 
The vector of absorptions along $m$ lines can then be approximated as
\begin{equation}
I_d \approx I_s \exp \left( -Gx \right),
\label{eq:tomo}
\end{equation} 
where ${I}_d$ is the vector of $m$ intensities at the detectors and $G=(g_{ij})$ is the $m \times n$ matrix of
intersection lengths for each of the $m$ lines. 
Even though the forward operator \eqref{eq:tomo} is nonlinear, 
the inference problem can be recast in a linear fashion by taking logarithm of both sides of \eqref{eq:tomo}.
This leads to the following linear  model for the inversion: ${y}={Gx}+\epsilon$, where the measurement vector is 
${y}=-\log({I}_d/I_s)$ and the measurement errors are assumed to be iid Gaussian,
$\epsilon \sim \Gauss(0,\sigma^2{I})$. 

The setup for the inference problem, borrowed from \cite{Heikkinen2008}, is as follows. The rectangular domain is discretized with an $n \times n$ grid. The true object consists of three circular inclusions, each of uniform density, inside an annulus. Ten X-ray sources are positioned on one side of a circle, and each source sends a fan of 100 X-rays that are measured by detectors on the opposite side of the object. Here, the 10 sources are distributed evenly so that they form a total illumination angle of 90 degrees, resulting in a limited-angle CT problem. We use the exponential model (\ref{eq:beer}) to generate synthetic data in a discretization-independent fashion by computing the exact intersections between the rays and the circular inclusions in the domain. Gaussian noise with standard deviation $\sigma=0.002$ is added to the simulated data. The imaging setup and data from one source are illustrated in Figure \ref{fig:tomodata}.

The unknown density is estimated on a $128 \times 128$ grid. Thus the discretized vector, $x$, has length $16384$, and direct computation of the posterior mean and the posterior covariance matrix, as well as generation of posterior samples, can be computationally nontrivial. To define the prior distribution, $x$ is modeled as a discretized solution of a stochastic PDE of the form:
\begin{equation} 
\label{eq:SPDE}
\gamma \left(  \kappa^2 \mathcal{I} -\triangle   \right) x(s) = \mathcal{W}(s), \qquad s\in\Omega,
\end{equation}
where $ \mathcal{W}$ is a white noise process, $\triangle$ is the Laplacian operator, and $\mathcal{I}$ is the identity operator. The solution of \eqref{eq:SPDE} is a Gaussian random field whose correlation length and variance are controlled by the free parameters $\kappa$ and $\gamma$, respectively. A square root of the prior precision matrix of $x$ (which is positive definite) can then be easily computed (see \cite{lindgren2011explicit} for details). We use $\kappa=10$ and $\gamma=\sqrt{800}$ in our simulations.

Our first task is to compute an optimal  approximation of $\Gpos$ as a low-rank negative update of $\Gpr$ (cf.\ Theorem \ref{Thm:main_th}).  Figure \ref{fig:tomosamp} (top row) shows the convergence of the approximate posterior variance as the rank of the update increases. The zero-rank update corresponds to $\Gpr$ (first column). For this formally $16384$-dimensional problem, a good approximation of the posterior variance is achieved with a rank $200$ update; hence the data are informative only on a low-dimensional subspace. The quality of the covariance matrix approximation is also reflected in the structure of samples drawn from the approximate posterior distributions (bottom row). All five of these samples are drawn using the same random seed and the exact posterior mean, so that all the differences observed are due to the approximation of $\Gpos$. Already with a rank $100$ update, the small-scale features of the approximate posterior sample match those of the exact posterior sample. In applications, agreement in this ``eye-ball norm'' is important.
\highlight{Of course, Theorem \ref{Thm:main_th} also provides an exact formula for the error in the posterior covariance; this error is shown in the right panel of Figure \ref{fig:tomoeigs} (blue curve).}

\begin{figure}[!htp]
\begin{center}
\includegraphics[width=0.45\textwidth]{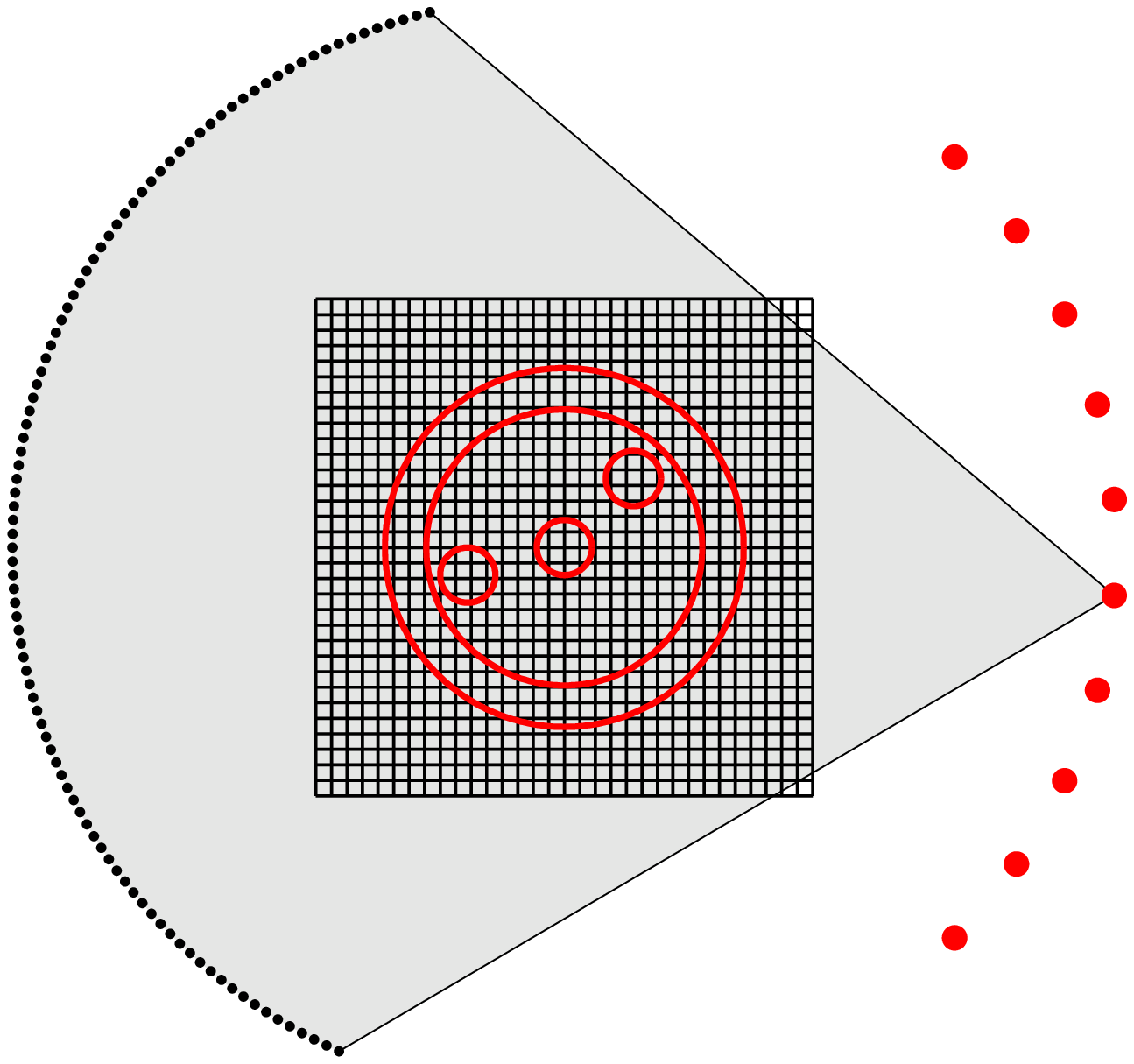}
\includegraphics[width=0.45\textwidth]{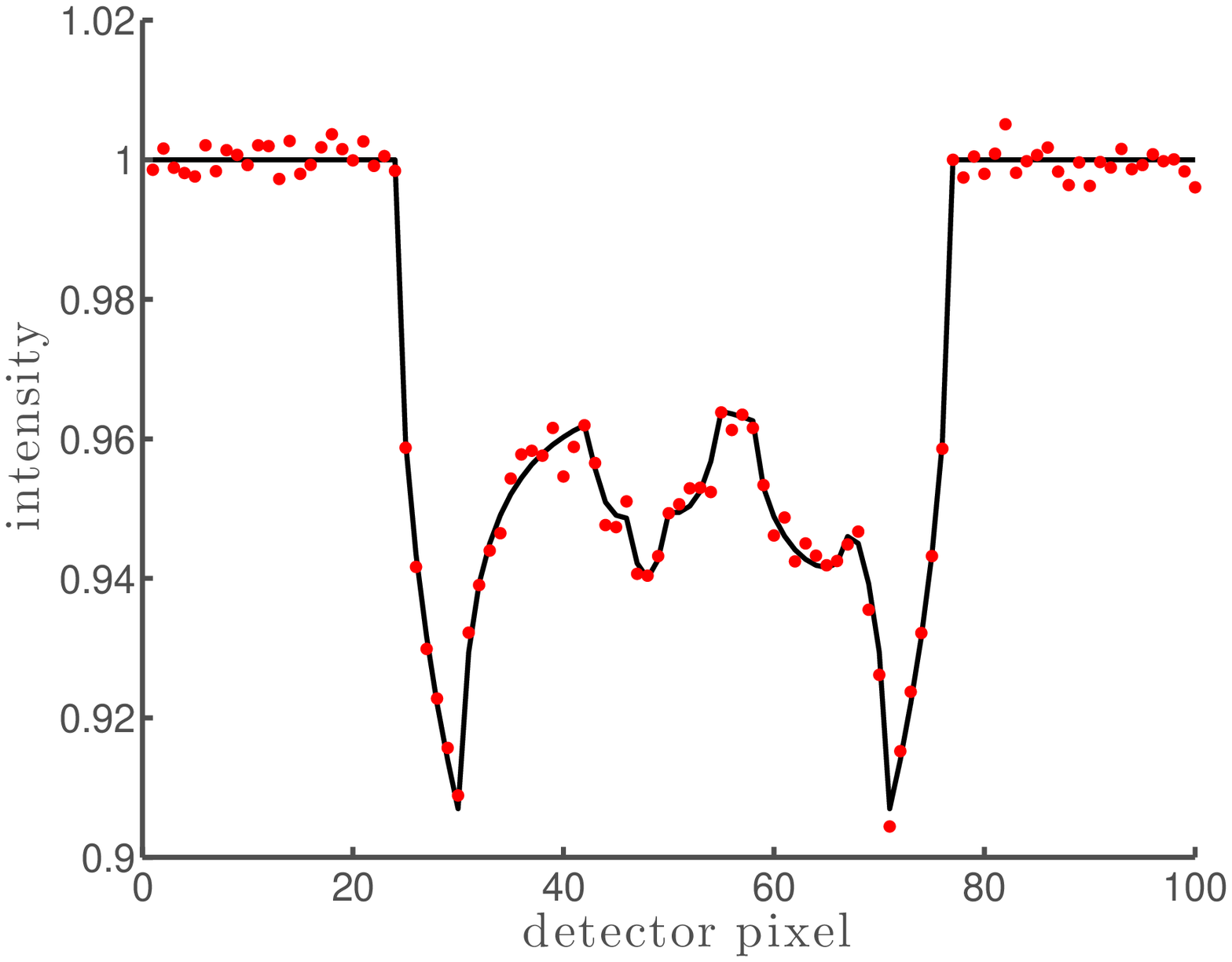}
\caption{X-ray tomography problem. Left: Discretized domain, true object, sources (red dots), and detectors corresponding to one source (black dots). The fan transmitted by one source is illustrated in gray. The density of the object is $0.006$ in the outer ring and $0.004$ in the three inclusions; the background density is zero. Right: The true simulated intensity (black line) and noisy measurements (red dots) for one source.} 
\label{fig:tomodata}
\end{center}
\end{figure}

\begin{figure}[!htp]
\begin{center}
\includegraphics[width=0.8\textwidth]{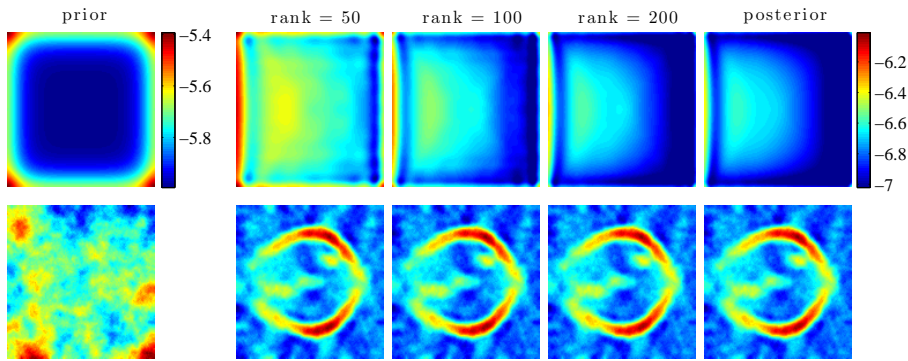}
\caption{X-ray tomography problem. First column:  Prior variance field, in log scale (top), and a sample drawn from the prior distribution (bottom). Second through last columns (left to right): Variance field, in log scale, of the approximate posterior as the rank of the update increases (top); samples from the corresponding approximate posterior distributions (bottom) assuming exact knowledge of the  posterior mean.}
\label{fig:tomosamp}
\end{center}
\end{figure}

Our second task is to assess the performances of the two optimal posterior mean approximations given in Section \ref{sec:mean}. We will use $\muposr$ to denote the low-rank approximation and $\muposhr$ to denote the low-rank \textit{update} approximation. Recall that both approximations are linear functions of the data $y$, given by $\muposr = A^\ast y$ with $A^\ast\in\Ac_r$ and $\muposhr = \widehat{A}^\ast y$ with $\widehat{A}^\ast \in\widehat{\Ac}_r$, where the classes $\Ac_r$ and $\widehat{\Ac}_r$ are defined in \eqref{eq:Aclass}. As in Section~\ref{sec:mean}, we shall use $\Ac$ to denote either of the two classes.

Figure \ref{fig:tomores} shows the normalized error $\| \mu(y) - \mupos\|_{\Gpos^{-1}} / \|   \mupos\|_{\Gpos^{-1}}$ for different approximations $\mu(y)$ of the true posterior mean $\mupos$ and a fixed realization $y$ of the data. The error is a function of the order $r$ of the approximation class $\Ac$. Snapshots of $\mu(y)$  are shown along the two error curves. For reference, $\mupos$ is also shown at the top. We see that the errors decrease monotonically, but that the low-rank approximation outperforms the low-rank update approximation for lower values of $r$. This is consistent with the discussion at the end of Section \ref{sec:mean}; the crossing point of the error curves is also consistent with that analysis. In particular, we expect the low-rank update approximation to outperform the low-rank approximation only when the approximation starts to include generalized eigenvalues of the pencil $(H,\Gpr^{-1})$ that are less than one---i.e., once the approximations are no longer \textit{under-resolved}. This can be confirmed by comparing Figure \ref{fig:tomores} with the decay of the generalized eigenvalues of the  pencil $(H,\Gpr^{-1})$ in the right panel of Figure \ref{fig:tomoeigs} (blue curve). 

\highlight{
On top of each snapshot in Figure \ref{fig:tomores}, we show the \textit{relative CPU time} required to compute the corresponding posterior mean approximation for each new realization of the data. The relative CPU time is defined as the time required to compute this approximation\footnote{This timing does not include the computation of \eqref{eq: mean_approx_lowrank} or \eqref{eq: mean_approx_fullrank}, which should be regarded as \textit{offline} steps. Here we report the time necessary to apply the optimal linear function to any new realization of the data, i.e., the \textit{online} cost.}
divided by the time required to apply the posterior precision matrix to a vector. This latter operation is essential to computing the posterior mean via an iterative solver, such as a Krylov subspace method. These solvers are a standard choice for computing the posterior mean in large-scale inverse problems. Evaluating the ratio allows us to determine how many solver iterations could be performed with a computational cost roughly equal to that of approximating the posterior mean for a new realization of the data.}
Based on the reported times, a few observations can be made. First of all, as anticipated in Section \ref{sec:mean}, computing $\muposr$ for any new realization of the data is faster than computing $\muposhr$. Second, obtaining an accurate posterior mean approximation requires roughly $r=200$, and the relative CPU times for this order of approximation are 7.3 for $\muposr$ and 29.0 for $\muposhr$. These are roughly the number of iterations of an iterative solver that one could take for equivalent computational cost. That is, the speedup of the posterior mean approximation compared to an iterative solver is not particularly dramatic in this case, because the forward model $A$ 
is simply a sparse matrix that is cheap to apply. This is different for the heat equation example discussed in Section~\ref{sec:heat}.

Note that the above computational time estimates exclude other costs associated with iterative solvers. For instance, preconditioners are often applied; these significantly decrease the number of iterations needed for the solvers to converge but, on the other hand, increase the cost per iteration. A popular approach for solving the posterior mean efficiently is to use the prior covariance as the preconditioner \cite{bui2012extreme}. In the limited-angle tomography problem, including the application of this  preconditioner in the reference CPU time would reduce the relative CPU time of our $r=200$ approximations to 0.48 for $\muposr$ and 1.9 for $\muposhr$. That is, the cost of computing our approximations is roughly equal to \textit{one iteration} of a prior-preconditioned iterative solver. The large difference compared to the case without 
preconditioning is due to the fact that the cost of applying the prior here is computationally much higher than applying the forward model.

Figure \ref{fig:tomoerror_hd_limited} (left panel) shows unnormalized errors in the approximation of $\mupos$,
\begin{equation}\label{eq:errmu}
\|  e (y) \|_{\Gpos^{-1}}^2 = \| \muposr - \mupos  \|^2_{\Gpos^{-1}}\quad\mbox{and}\quad
\|  \widehat{e}  (y) \|_{\Gpos^{-1}}^2 = \| \muposhr - \mupos  \|^2_{\Gpos^{-1}},
\end{equation}
for the same realization of $y$ used in Figure \ref{fig:tomores}. In the same panel we also show the expected values of these errors over the prior predictive distribution of $y$, which is exactly the $r$-dependent component of the Bayes risk given in Theorems \ref{thm:mean_approx_lowrank} and \ref{thm:mean_approx_fullrank}. Both sets of errors decay with increasing $r$ and show a similar crossover between the two approximation classes. But the particular error $\|  e (y) \|_{\Gpos^{-1}}^2$ departs consistently from its expectation; this is not unreasonable in general (the mean estimator has a nonzero variance), but the offset may be accentuated in this case because the data are generated from an image that is not drawn from the prior. (The right panel of Figure \ref{fig:tomoerror_hd_limited}, which comes from Example 3, represents a contrasting case.)

By design, the posterior approximations described in this paper perform well when the data inform a low-dimensional subspace of the parameter space. 
To better understand this effect, we also consider a \textit{full-angle} configuration of the tomography problem, wherein the sources and detectors are evenly spread around the entire unknown object. In this case, the data are more informative than in the limited-angle configuration. This can be seen in the decay rate of the generalized eigenvalues of the pencil $(H,\Gpr^{-1})$ in the center panel of Figure \ref{fig:tomoeigs} (blue and red curves); eigenvalues for the full-angle configuration decay more slowly than for the limited-angle configuration. Thus, according to the optimal loss given in \eqref{eq:error_estimate} (Theorem \ref{Thm:main_th}), the prior-to-posterior update in the full-angle case must be of greater rank than the update in the limited-angle case for any given approximation error.  Also, good approximation of $\mupos$ in the full-angle case requires higher order of the approximation class $\Ac$, as is shown in Figure \ref{fig:tomomean_hd_full}. But because the data are strongly informative, they allow an almost perfect reconstruction of the underlying truth image. The relative CPU times are similar to the limited angle case: roughly 8 for $\muposr$ and 14 for $\muposhr$. If preconditioning with the prior covariance is included in the reference CPU time calculation, the relative CPU times drop to 1.5 for $\muposr$ and to 2.6 for $\muposhr$.
\highlight
{
We remark that in realistic applications of X-ray tomography, the limited angle setup is extremely common as it is cheaper and more flexible (yielding smaller and lighter devices) than a full-angle configuration.
}

\begin{figure}[!htp]
\begin{center}
\includegraphics[width=0.8\textwidth]{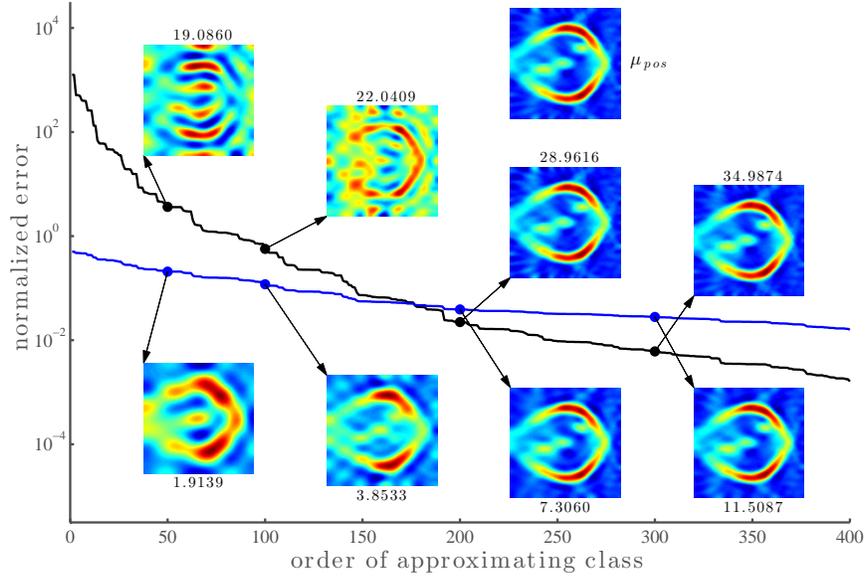}
\caption{Limited-angle X-ray tomography: Comparison of the optimal posterior mean approximations, $\muposr$ (blue) and $\muposhr$ (black) of
$\mupos$ for a fixed realization of the data $y$, as a function of the order $r$ of the approximating
classes $\Ac_r$ and $\widehat{\Ac}_r$, respectively.
The normalized error for an approximation $\mu(y)$ is defined as 
    $  \left\| \mu(y) - \mupos   \right\|_{\Gpos^{-1}} /  
    \left\|   \mupos \right\|_{\Gpos^{-1}}$. The numbers above or below the snapshots indicate the relative CPU time of the corresponding mean approximation---i.e., the time required to compute the approximation divided by the time required to apply the posterior precision matrix to a vector.}
\label{fig:tomores}
\end{center}
\end{figure}

\begin{figure}[!htp]
\begin{center}
\includegraphics[width=0.35\textwidth]{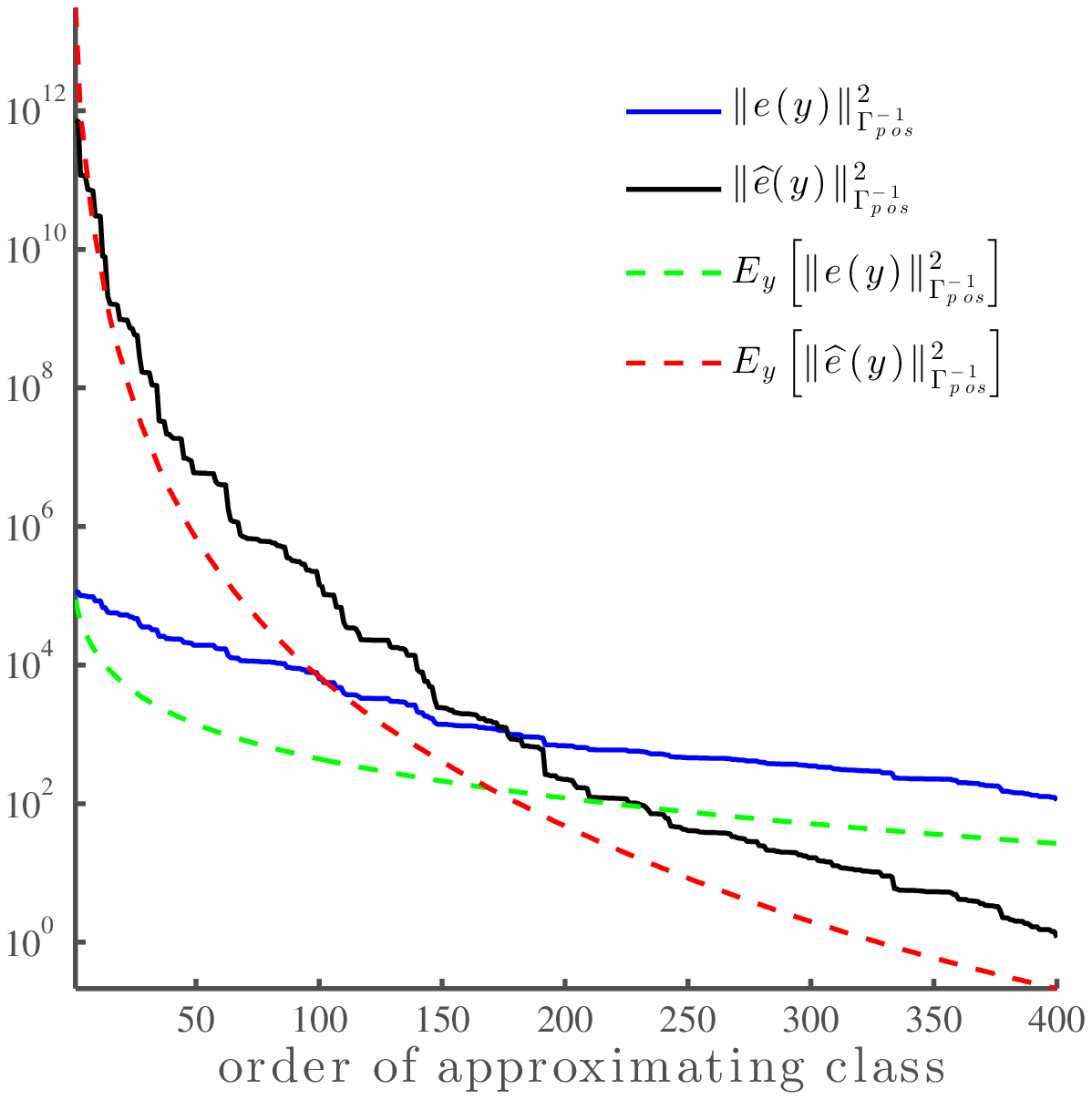} \includegraphics[width=0.35\textwidth]{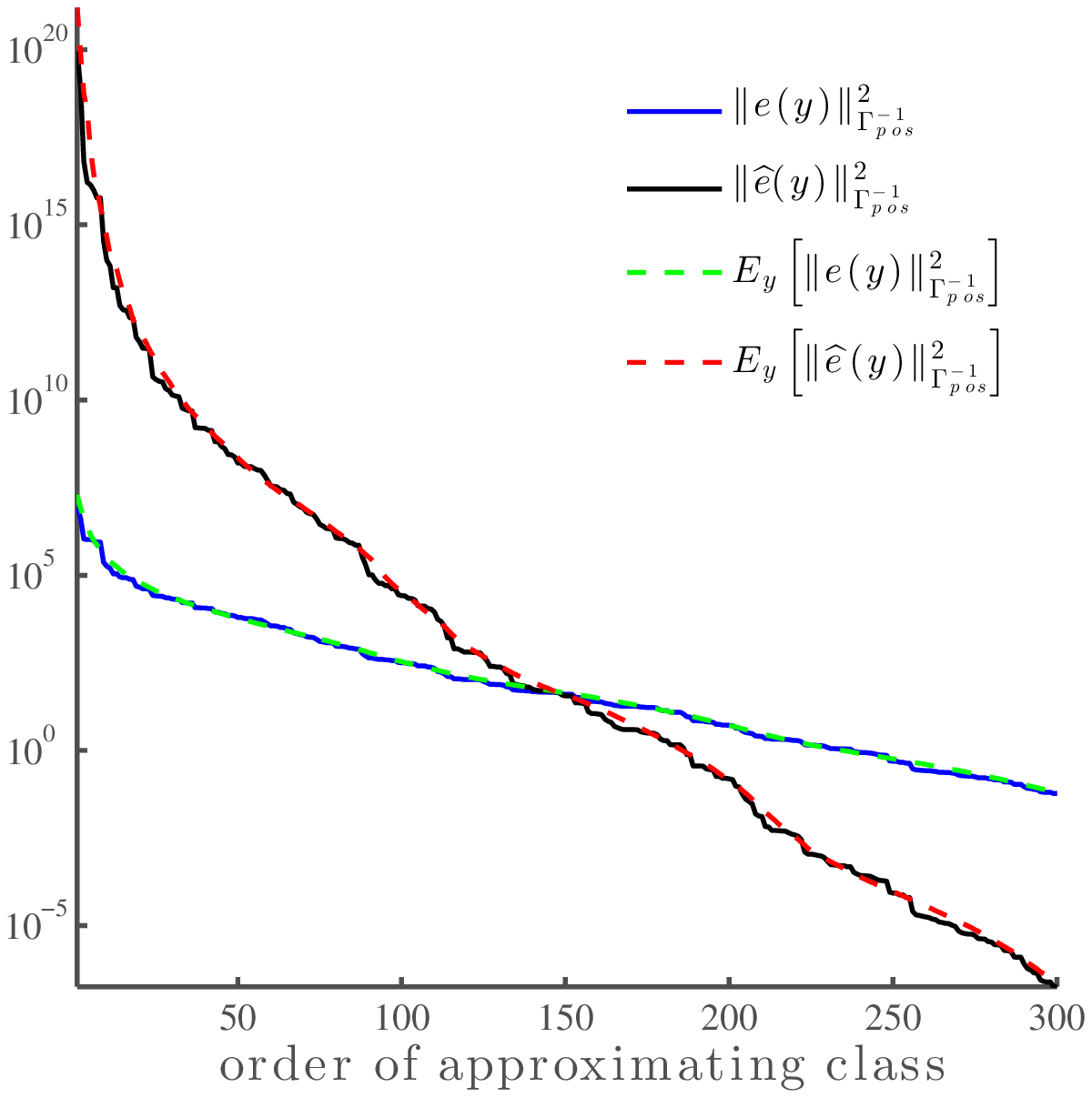}
\caption{The errors $\|  e(y) \|_{\Gpos^{-1}}^2$ (blue) and  $\|\widehat{e}(y) \|_{\Gpos^{-1}}^2$ (black)
defined by \eqref{eq:errmu}, and their expected values in green and red, respectively; for Example
2 (left panel) and Example 3 (right panel).}
\label{fig:tomoerror_hd_limited}
\end{center}
\end{figure}

\begin{figure}[!htp]
\begin{center}
\includegraphics[width=0.6\textwidth]{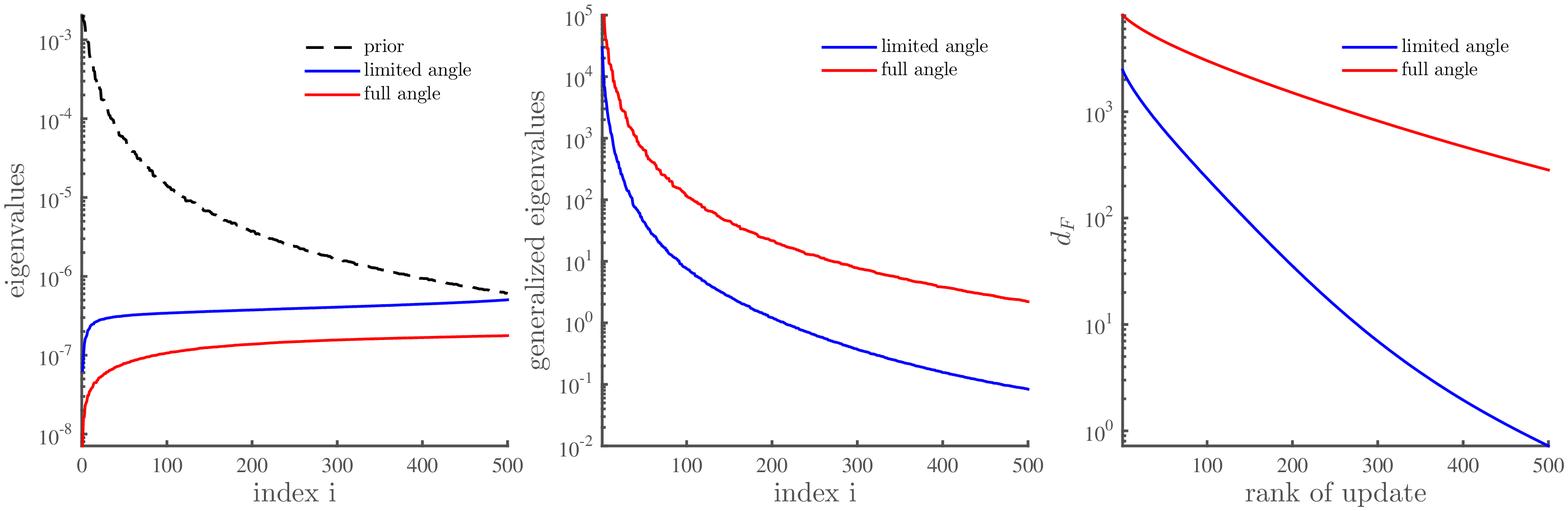}
\caption{
\highlight{
Left: Leading eigenvalues of $\Gpr$ and $H^{-1}$  in the limited-angle and full-angle X-ray tomography problems. Center: Leading generalized eigenvalues of the pencil $(H,\Gpr^{-1})$ in the limited-angle (blue) and full-angle (red) cases. Right: $\df (\Gpos, \Gposh)$ as a function of the rank of the update $KK^\top$, with $\Gposh = \Gpr - KK^\top$, in the limited-angle (blue) and full-angle (red) cases.
}
}
\label{fig:tomoeigs}
\end{center}
\end{figure}

\begin{figure}[!htp]
\begin{center}
\includegraphics[width=0.8\textwidth]{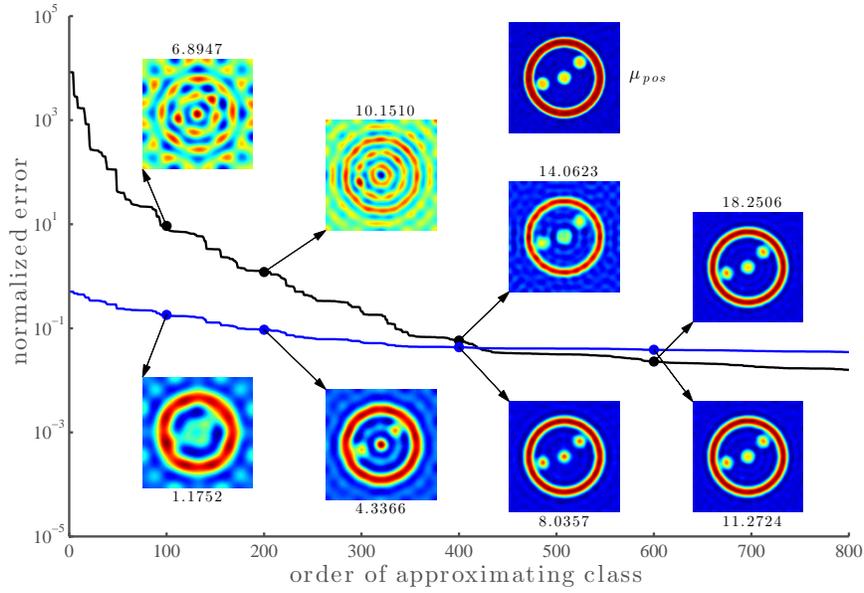}
\caption{Same as Figure \ref{fig:tomores}, but for full-angle X-ray tomography (sources and receivers spread uniformly around the entire object).}
\label{fig:tomomean_hd_full}
\end{center}
\end{figure}

\subsection{Example 3: Heat equation}
\label{sec:heat}

Our last example is the classic linear inverse problem of solving for the initial conditions of an inhomogeneous heat equation.
Let $u(s, t)$ be the time dependent state of the heat equation on $s = (s_1, s_2)\in \Omega = [0, 1]^2$, $t\geq 0$,
and let $\kappa(s)$ be the heat conductivity field.
Given initial conditions, $u_0(s) = u(s, 0)$, the state evolves in time according to the linear heat equation:
\begin{eqnarray} 
\frac{\partial u(s, t)}{\partial t} & = &- \nabla \cdot (\kappa(s) \nabla u(s, t)), \qquad s \in \Omega,\,\, t > 0, \nonumber \\
\kappa(s) \nabla u(s, t) \cdot {n}({s}) & = & 0,   \qquad {s} \in \partial\Omega,\,\, t > 0,
\label{eq:heat}
\end{eqnarray}
where ${n}({s})$ denotes the outward-pointing unit normal at $s\in\partial \Omega$. 
We place $n_s = 81$  sensors at the locations $s_1,\ldots, s_{n_s}$, uniformly spaced within the lower left quadrant of the spatial domain, as illustrated by the black dots in Figure \ref{fig:heat_setup}. We use a finite-dimensional discretization of the parameter space based on the finite element method on a 
regular $100\times100$ grid, $\{s_i^\prime\}$. Our goal is to infer the vector $x=(u_0(s_i^\prime))$ of initial conditions 
on the grid. Thus, the dimension of the parameter space for the inference problem is $n=10^4$.
We use data measured at $50$ discrete times $t = t_1, t_2, \ldots, t_{50}$, where $t_i = i \triangle t$, and $\triangle t = 2 \times 10^{-4}$.
At each time $t_i$, pointwise observations of the state $u$ are taken at these sensors, i.e., 
\begin{equation} \label{eq:obs_operator}
 d_i = \mathcal{C} u({s}, t_i),
\end{equation}
where $\mathcal{C}$ is the observation operator that maps the function $u(s,t_i)$ to $d=(u(s_1,t_i),\ldots,u(s_n,t_i))^\top$.
The vector of observations is then 
\(
d = [d_{1} ; d_{2} ; \ldots ; d_{50}].
\)
The noisy data vector is $y= d + \varepsilon$, where $\varepsilon \sim \Gauss(0,\sigma^2 I)$ and $\sigma = 10^{-2}$. 
Note that the data are a linear function of the initial conditions, perturbed by Gaussian noise. Thus the data can be written as:
\begin{equation}
   y = G x + \varepsilon, \qquad \varepsilon \sim \Gauss( 0 , \sigma^2 I ).
\end{equation}
where $G$ is a linear map defined by the composition of the forward model \eqref{eq:heat} with the observation operator \eqref{eq:obs_operator}, both linear. 

We generate synthetic data by evolving the initial conditions shown in Figure \ref{fig:heat_setup}. This ``true'' value of the inversion parameters $x$ is a discretized realization of a Gaussian process satisfying an SPDE of the same form used in the previous tomography example, but now with a non-stationary permeability field. In other words, the truth is a draw from the prior in this example (unlike in the previous example), and the prior Gaussian process satisfies the following SPDE:
\begin{equation}
\gamma \left(  \kappa^2 \mathcal{I} - \nabla \cdot {\bf c}(s)\nabla \  \right) x(s) = \mathcal{W}(s) \qquad s\in\Omega,
\end{equation}
where ${\bf c}(s)$ is the space-dependent permeability tensor. 

Figure \ref{fig:heat_mean} and the right panel in Figure \ref{fig:tomoerror_hd_limited} show our numerical results. They have the same interpretations as Figures \ref{fig:tomores} and \ref{fig:tomoerror_hd_limited} in the tomography example. The trends in the figures are consistent with those encountered in the previous example and confirm the good performance of the optimal low-rank approximation. 
Notice that in Figures \ref{fig:heat_mean} and \ref{fig:tomoerror_hd_limited} the approximation of the posterior mean appears to be nearly perfect (visually) once the error curves for the two approximations cross. This is somewhat expected from the theory since we know that the crossing point should occur when  the approximations start to use eigenvalues of the \highlight{pencil $(H,\Gpr^{-1})$} that are less than one---that is, once we have exhausted directions in the parameter space where the data are more constraining than the prior.

Again, we report the relative CPU time for each posterior mean approximation above/below the corresponding snapshot in Figure \ref{fig:heat_mean}. The results differ significantly from the tomography example. For instance, at order $r=200$, which yields approximations that are visually indistinguishable from the true mean, the relative CPU times are 0.001 for $\muposr$ and 0.53 for $\muposhr$. 
\highlight{Therefore we can compute an accurate mean approximation for a new realization of the data much more quickly than taking one iteration of an iterative solver. Recall that, consistent with the setting described at the start of Section~\ref{sec:mean}, this is a comparison of \textit{online} times, after the matrices \eqref{eq: mean_approx_lowrank} or \eqref{eq: mean_approx_fullrank} have been precomputed.}
The difference between this case and tomography example of Section~\ref{sec:tomo} is due to the higher CPU cost of applying the forward and adjoint models for the heat equation---solving a time dependent PDE versus applying a sparse matrix. Also, because the cost of applying the prior covariance is negligible compared to that of the forward and adjoint solves in this example, preconditioning the iterative solver with the prior would not strongly affect the reported relative CPU times, unlike the tomography example.

Figure \ref{fig:heat_eigens} illustrates some important directions characterizing the heat equation inverse problem.  The first two columns show the four leading eigenvectors of, respectively, $\Gpr$ and $H$. Notice that the support of the eigenvectors of $H$ concentrates around the sensors. The third column shows the four leading directions $(\widehat{w}_i)$ defined in Theorem \ref{Thm:main_th}. These directions define the optimal prior-to-posterior covariance matrix update (cf.\ \eqref{minimizer_theorem}). This update of $\Gpr$ is \textit{necessary} to capture directions $(\widetilde{w}_i)$ of greatest relative difference between prior and posterior variance (cf.\ Corollary \ref{cor:Minv}). The four leading directions $(\widetilde{w}_i)$ are shown in the fourth column. The support of these modes is again concentrated around the sensors, which intuitively makes sense as these are directions of greatest variance reduction.

\begin{figure}[!htp]
\begin{center}
\includegraphics[width=0.7\textwidth]{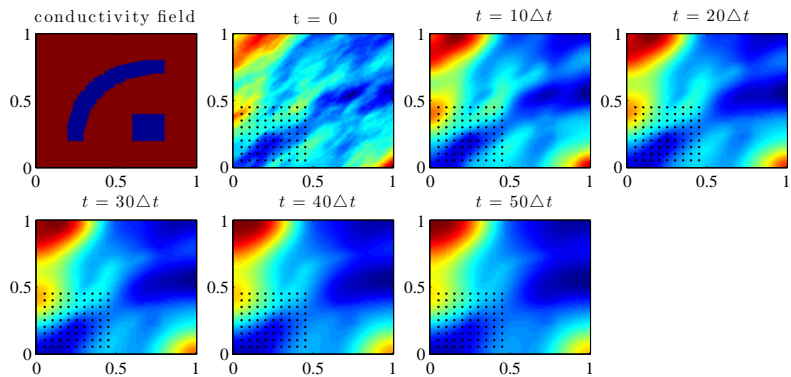}
\caption{Heat equation (Example 3). Initial condition (top left) and several snapshots of the states at different times. Black dots indicate sensor locations.}
\label{fig:heat_setup}
\end{center}
\end{figure}

\begin{figure}[!htp]
\begin{center}
\includegraphics[width=0.8\textwidth]{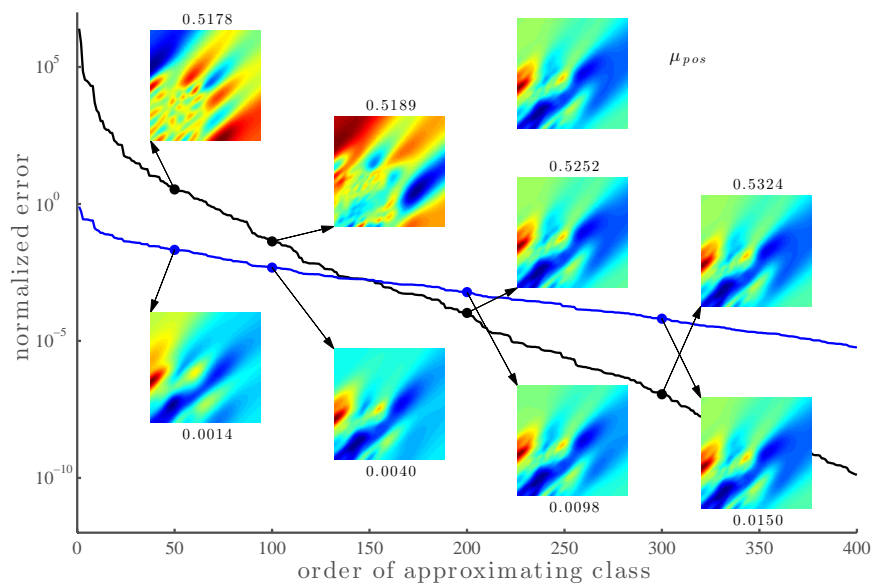}
\caption{Same as Figure~\ref{fig:tomores}, but for Example 3 (initial condition inversion for the heat equation).}
\label{fig:heat_mean}
\end{center}
\end{figure}

\begin{figure}[!htp]
\begin{center}
\includegraphics[width=0.4\textwidth]{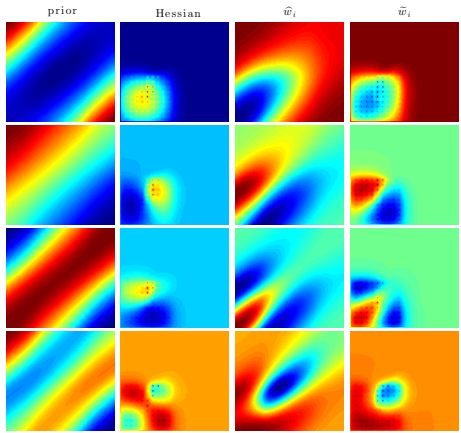}
\caption{Heat equation (Example 3). First column: Four leading eigenvectors of $\Gpr$. Second column: Four leading eigenvectors of 
$H$. Third column: Four leading directions $(\widehat{w}_i)$ (cf.\ \eqref{minimizer_theorem}).
 Fourth column: Four leading directions $(\widetilde{w}_i)$ (cf.\ Corollary \ref{cor:Minv})}
\label{fig:heat_eigens}
\end{center}
\end{figure}

\section{Conclusions}
\label{s:conclusions}

This paper has presented and characterized optimal approximations of the Bayesian solution of linear inverse problems, with Gaussian prior and noise distributions defined on finite-dimensional spaces. In a typical large-scale inverse problem, observations may be informative---relative to the prior---only on a low-dimensional subspace of the parameter space. Our approximations therefore identify and exploit low-dimensional structure in the \textit{update} from prior to posterior. 

We have developed two types of optimality results. In the first, the posterior covariance matrix is approximated as a low-rank negative semidefinite update of the prior covariance matrix. We describe an update of this form that is optimal with respect to a broad class of loss functions between covariance matrices, exemplified by the  F\"{o}rstner metric \cite{forstner2003metric} for symmetric positive definite matrices. We argue that this is the appropriate class of loss functions with which to evaluate approximations of the posterior covariance matrix, and show that optimality in such metrics identifies directions in parameter space along which the posterior variance is reduced the most, relative to the prior.  
Optimal low-rank updates are derived from a generalized eigendecomposition of the pencil defined by the negative log-likelihood Hessian and the prior precision matrix. These updates have been proposed in previous work \cite{flath2011fast}, but our work complements these efforts by characterizing the optimality of the resulting approximations.
Under the assumption of exact knowledge of the posterior mean, our results extend to optimality statements between the associated distributions (e.g., optimality in the Hellinger distance and in the Kullback-Leibler divergence).
Second, we have developed fast approximations of the posterior mean that are useful when repeated evaluations thereof are required for multiple realizations of the data (e.g., in an online inference setting). These approximations are optimal in the sense that they minimize the Bayes risk for squared-error loss induced by the posterior precision matrix. The most computationally efficient of these approximations expresses the posterior mean as the product of a single low-rank matrix with the data.
We have demonstrated the covariance and mean approximations numerically on a variety  of inverse problems: synthetic problems constructed from random Hessian and prior covariance matrices; an X-ray tomography problem with different observation scenarios; and inversion for the initial condition of a heat equation, with localized observations and a non-stationary prior.

This work has several possible extensions of interest, some of which are already part of ongoing research. First, it is natural to generalize the present approach to infinite-dimensional parameter spaces endowed with Gaussian priors. This setting is essential to understanding and formalizing Bayesian inference over function spaces \cite{bui2013computational,stuart2010inverse}. Here, by analogy with the current results, one would expect the posterior covariance operator to be well approximated by a finite-rank negative perturbation of the prior covariance operator. A further extension could allow the data to become infinite-dimensional as well. 
Another important task is to generalize the present methodology to inverse problems with nonlinear forward models. One approach for doing so is presented in \cite{cui2014likelihood}; other approaches are certainly possible. 
Yet another interesting research topic is the study of analogous approximation techniques for sequential inference.  We note that the assimilation step in a linear (or linearized) data assimilation scheme can be already tackled within the framework presented here. But the nonstationary setting, where inference is interleaved with evolution of the state, introduces the possibility for even more tailored and structure-exploiting approximations. %

\section*{Acknowledgments}
This work was supported by the US Department of Energy, Office of Advanced Scientific Computing (ASCR), under grant numbers DE-SC0003908 and DE-SC0009297. We thank Jere Heikkinen from Bintec Ltd.\ for providing us with the code used in Example 2.
\highlightTwo{
We also thank Julianne and Matthias Chung for many insightful discussions.
}

\appendix
\section{Technical results}\label{sec:proofs}
Here we collect the proofs and other technical results necessary to support the statements made in the previous sections.
\medskip	

We start with an auxiliary approximation result that plays an important role in our analysis. Given a semi-positive definite diagonal matrix $D$, we seek an approximation of $D+I$  by a rank $r$ perturbation of the identity, $ U U^\top+I$, that minimizes a loss function from the class $\Lc$ defined in \eqref{eq:form_loss_fun}. The following lemma shows that the optimal solution $\Uh\Uh^\top$ is simply the best rank $r$ 
approximation of the matrix $D$ in the Frobenius norm. 
 
\medskip

\begin{lemma}[Approximation lemma] \label{lemma:diagcase}
Let $D=\mathrm{diag}\{d_1^2,\ldots,d_n^2\}$, with $d_{i}^2\geq d_{i+1}^2$, and $L\in \Lc.$ 
Define the functional $\Jc:\R^{n\times r}\to \R$, as:
\(
\Jc(U) = L( UU^\top + I , D + I)=\sum_i f(\sigma_i),
\)
where $(\sigma_i)$ are the generalized eigenvalues of the pencil $(UU^\top+I,D+I)$ and $f\in\Uc$.
Then:
\begin{itemize}
\item[(i)]  There is a minimizer, $\Uh$, of $\Jc$ such that
\begin{equation}\label{eq:Uhat}
\Uh\Uh^\top = \sum_{i=1}^r d_i^2\, e_i e_i^\top.
\end{equation}
 where $(e_i)$ are the columns of the identity matrix.
\item[(ii)]  If the first $r$ eigenvalues of $D$ are distinct, then any minimizer of $\Jc$
satisfies \eqref{eq:Uhat}.
\end{itemize}
 \end{lemma}
\smallskip
\begin{proof}
 The idea is to apply \cite[Theorem 1.1]{lewis1996derivatives} to the functional $\Jc$. To this end, we notice that $\Jc$ can be equivalently written as:
\(
\Jc(U) = F\circ \rho_n\circ g(U),
\) 
where: $F:\R_+^n\to \R $ is of the form $F(x) = \sum_{i=1}^n f(x_i)$; $\rho_n$  denotes a 
function that maps an $n\times n$ SPD matrix $A$ to its eigenvalues $\sigma =(\sigma_i)$ (i.e., $\rho_n(A) = \sigma$ and since $F$ is a symmetric function, the order of the eigenvalues is  
irrelevant);
and the mapping $g$ is given by:
\(
g(U) = (D+I)^{-1/2}(U U^\top + I)(D+I)^{-1/2},
\)
for all $U\in\R^{n\times r}$.  Since the function $F\circ\rho_n$ satisfies the 
hypotheses in  \cite[Theorem 1.1]{lewis1996derivatives}, 
$F\circ\rho_n$ is differentiable at the SPD matrix $X$ if and only if 
$F$ is differentiable at $\rho_n(X)$, in which case
\(
(F\circ \rho_n)'(X) = Z S_\sigma Z^\top,
\)
where 
\[
S_\sigma=\mathrm{diag}[\,F'(\rho_n(X))\,] = \mathrm{diag}\{f'(\sigma_1),\ldots,f'(\sigma_n)\},
\] 
and $Z$ is an orthogonal matrix such that $X=Z\, \mathrm{diag}[\,\rho_n(X)\,]Z^\top$. Using the
chain rule, we obtain
\[
\frac{\partial \Jc(U) }{\partial\, U_{ij}} = 
\trace\left(Z S_\sigma Z^\top \,\frac{\partial g(U)}{\partial\,U_{ij}}\right),
\]
which leads to the following gradient of $\Jc$ at $U$:
\[
\Jc'(U) = 2(D+I)^{-1/2}ZS_\sigma (D+I)^{-1/2}Z^\top U= 2 \,W S_\sigma W^\top U,
\]
where the orthogonal matrix $Z$ is such that the matrix $W=(D+I)^{-1/2}Z$ satisfies
\begin{equation}\label{eq:Weq}
(U U^\top +I)W=(D +I)W \Upsilon_\sigma
\end{equation}
with $\Upsilon_\sigma = \mathrm{diag}(\sigma)$. 
Now we show that the functional $\Jc$ is coercive. Let $(U_k)$ be a sequence of matrices
such that $\|U_k\|_F \rightarrow \infty$. Hence, $ \sigma_{max}( g(U_k))\rightarrow\infty$ and so does $\Jc$  
since:
\[
 \Jc(U_k)\ge f(\sigma_{max}( g(U_k))) + (n-1) f(1) 
 \]
and $f(x)\rightarrow\infty$ as $x\rightarrow\infty$. Thus, $\Jc$ is a differentiable coercive functional,
and has a global minimizer $\Uh$ with zero gradient: 
\begin{equation}\label{eq:optcond1}
\Jc'(\Uh)=2 W S_\sigma W^\top \Uh=0.
\end{equation} 
However, since $f\in \Uc$, $f'(x)=0$ iff $x=1$. It follows that condition \eqref{eq:optcond1}
is equivalent to
\begin{equation}\label{eq:optcond2}
(I-\Upsilon_\sigma)W^\top \Uh= 0.
\end{equation}
\eqref{eq:Weq} and \eqref{eq:optcond2} give
\(
\Uh\Uh^\top - D = W^{-\top}\Upsilon^{-1}(\Upsilon-I)W^\top,
\)
and right-multiplication by $\Uh\Uh^\top$ then yields:
\begin{equation}\label{eq:optcond3}
D\,(\Uh\Uh^\top) = (\,\Uh\Uh^\top)^2.
\end{equation} 
 In particular,
if $u$ is an eigenvector of $\Uh\Uh^\top$ with nonzero eigenvalue $\alpha$, then $u$ is an eigenvector
of $D$, $Du = \alpha u$, and thus $\alpha=d_i^2>0$ for some $i$. Thus, any solution of \eqref{eq:optcond3} is such that:
\begin{equation}\label{eq:UhUht}
\Uh\Uh^\top = \sum_{i=1}^{r_k} d_{k_i}^2 e_{k_i}e_{k_i}^\top,
\end{equation}
for some subsequence $(k_\ell)$ of $\{1,\ldots,n\}$  and rank $r_k\le r$. Notice that any $\Uh$ satisfying \eqref{eq:optcond3} is also a critical point according to \eqref{eq:optcond2}.
From \eqref{eq:UhUht} we also find that
$g(\Uh)$ is a diagonal matrix,
\[
g(\Uh) = (D+I)^{-1}\left(\sum_{i=1}^{r_k} d_{k_i}^2 e_{k_i}e_{k_i}^\top + I\right).
\]
The diagonal entries $\sigma_i$, which are the eigenvalues of $g(\Uh)$, are 
given by $\sigma_i = 1$ if $i=k_\ell$ for some $\ell\le r_k$, or $\sigma_i=1/(1+d_i^2)$ otherwise.
In either case, we have $0<\sigma_i\leq 1$ and the monotonicity of $f$ implies that
$\Jc(\Uh)$ is minimized by the subsequence $k_1=1,\ldots,k_r=r$, and by the choice $r_k=r$. This proves \eqref{eq:Uhat}.
It is clear that if the first $r$ eigenvalues of $D$ are distinct, then any minimizer of $\Jc$
satisfies \eqref{eq:Uhat}.
\end{proof}

Most of the objective functions we consider have the same  structure as the loss
function $\Jc$. Hence, the importance of Lemma \ref{lemma:diagcase}.
\medskip

The next lemma shows that searching for a negative update of $\Gpr$ is 
equivalent to looking for a positive update of the prior precision matrix. In particular, the lemma provides a bijection between the two approximation classes,  ${\Mc}_r$ and $ {\Mc}_r^{-1}$, defined by \eqref{eq:class_M} and \eqref{eq:class_Minv}. In what follows, $\Gprs$ is any square root of the prior covariance matrix such that $\Gpr=\Gprs\,\Gprst$.
\medskip

\begin{lemma}[Prior updates]\label{lemma_low_rank}
For any negative semidefinite update of $\Gpr$, 
$\Gposh = \Gpr - KK^\top$ with $\Gposh\succ0$, 
there is a matrix $U$ (of the same rank as $K$) such that
\(
\Gposh = \left(\Gpr^{-1} + UU^\top\right)^{-1}.
\)
The converse is also true.
\end{lemma}
\begin{proof}
Let $Z D Z^\top=\Gprs^{-1}KK^\top\Gprs^{-\top} $, $D=\text{diag}\{d_i^2\}$, be a reduced SVD 
of $\Gprs^{-1}KK^\top\Gprs^{-\top} $. Since $\Gposh\succ0$ by assumption, we must have
$d_i^2<1$ for all $i$, and we may thus define $U=\Gprs^{-\top}ZD^{1/2}(I-D)^{-1/2}$.
By Woodbury's identity:
\begin{eqnarray*}
\left(\Gpr^{-1} + UU^\top\right)^{-1} &=& \Gpr-\Gpr U\left(I + U^\top\Gpr^{-1}U  \right)^{-1}U^\top\Gpr
 =  \Gpr -KK^\top = \Gposh.
	\end{eqnarray*}
Conversely, given a matrix $U$, we use again Woodbury's identity to write $\Gposh$
as a negative semidefinite update of $\Gpr$: $\Gposh=\Gpr -KK^\top\succ0$.  
\end{proof}     

\medskip
Now we prove our main result on approximations of the posterior covariance matrix.

\noindent
\textbf{Proof of Theorem \ref{Thm:main_th}.} Given a loss function $L\in\Lc$, our goal is to minimize:
\begin{equation} \label{eq:objective_min}
  L(\Gpos,\Gposh)=\sum_i f\left( \sigma_i \right)
\end{equation}
over $K\in\R^{n\times r}$ subject to the constraint $\Gposh=\Gpr-KK^\top \succ 0$, where $(\sigma_i)$ are the generalized eigenvalues  of the pencil $(\Gpos,\Gposh)$ and $f$ belongs to the class $\,\Uc$ defined by Eq. \eqref{eq:Uclass}. We also write $\sigma_i(\Gpos,\Gposh)$ to specify the pencil  corresponding to the eigenvalues.
By Lemma \ref{lemma_low_rank}, the optimization problem is equivalent to finding   a 
matrix, $U\in\R^{n\times r}$, that minimizes \eqref{eq:objective_min} subject 
to $\Gposh^{-1} = \Gpr^{-1} +UU^\top$.  
Observe that $(\sigma_i)$ are also the eigenvalues of the 
pencil   $(\Gposh^{-1},\Gpos^{-1})$.
\medskip

Let $WDW^\top=\Gprst H\,\Gprs$ with $D=\mathrm{diag}\{\delta_i^2\}$, be an SVD of 
$\Gprst H\,\Gprs$. Then,
by the invariance properties of the generalized eigenvalues we have: 
\begin{eqnarray*}
\sigma_i(\Gposh^{-1},\, \Gpos^{-1}) &=& 
\sigma_i(\, W^{\top}\Gprst\,\Gposh^{-1}\,\Gprs W \,,\, W^{\top}\Gprst\,\Gpos^{-1}\, \Gprs W\,)
 =  \sigma_i(\, Z Z^\top + I , \, D + I\,),
\end{eqnarray*} 
where $Z = W^\top\Gprst U$. Therefore, our goal reduces to finding a matrix, $Z\in\R^{n\times r}$,
that minimizes \eqref{eq:objective_min} with $(\sigma_i)$ being the generalized eigenvalues of the 
pencil $(\, Z Z^\top + I , \, D + I\,)$. 
Applying Lemma \ref{lemma:diagcase} leads to the simple solution:
\(
Z Z^\top = \sum_{i=1}^r \delta_i^2\,e_i e_i^\top, 
\)
where $(e_i)$ are the columns of the identity matrix. In particular, the solution is unique if the first $r$ 
eigenvalues of $\Gprst H\,\Gprs$ are distinct. The corresponding approximation $U U^\top$ is then
\begin{equation}\label{eq:UUt}
U U^\top =\Gprs^{-\top}\,WZ Z^\top W^\top \Gprs^{-1}=\sum_{i=1}^r\delta_i^2 \,\widetilde{w}_i\widetilde{w}_i^\top,
\end{equation}
where $\widetilde{w}_i = \Gprs^{-\top}w_i$ and $w_i$ is the $i$th column of $W$. 
Woodbury's identity gives the corresponding negative update of $\Gpr$ as:   
\begin{equation}  
     \Gposh=\Gpr-KK^\top, \quad KK^\top=\sum_{i=1}^r \delta_i^2 \left( 1 + \delta_i^2 \right)^{-1} \widehat{w_i} \widehat{w_i}^\top
  \end{equation}
with $\widehat{w_i}=\Gprs w_i$. 
\highlight{
Now, it suffices to note that the couples $(\delta_i^2,\widehat{w}_i)$
defined here are precisely the generalized eigenpairs of the pencil $(H,\Gpr^{-1})$.
}
At optimality, $\sigma_i=1$ for $i\le r$ and 
$\sigma_i=(1+\delta_i^2)^{-1}$ for $i>r$, proving \eqref{eq:error_estimate}. $\square$
\medskip

Before proving Lemma \ref{lemma:equivalence}, we recall that the Kullback-Leibler (K-L) 
divergence and the Hellinger distance between two multivariate Gaussians, $\nu_1=\Gauss(\mu, \Sigma_1)$ and 
$\nu_2=\Gauss(\mu, \Sigma_2)$, with the same mean and full rank covariance matrices are given, respectively, by \cite{pardo2005statistical}: 
\begin{equation}\label{eq:KL_formula}
 D_{\rm KL} \left( \nu_1 \|\, \nu_2 \right)= 
 \frac{1}{2} \left[ {\rm trace}\left(\Sigma_2^{-1}\Sigma_1\right) -{\rm rank}(\Sigma_1) - {\rm ln} \left( \frac{ {\rm det}(\Sigma_1)  }{ {\rm det}(\Sigma_2)    }\right) \right]
\end{equation}

\begin{equation}\label{eq:helldet}
 d_{\mathrm{Hell}}\left( \nu_1,\nu_2 \right)=
\sqrt{1-\frac{ |\Sigma_1|^{1/4}\,|\Sigma_2|^{1/4} }  { | \frac{1}{2}\Sigma_1 + \frac{1}{2}\Sigma_2 |^{1/2}  }}.
\end{equation}

\medskip\noindent
{\bf Proof of Lemma \ref{lemma:equivalence}.}
By \eqref{eq:KL_formula}, the K-L divergence between the posterior $\nupos(y)$ and the Gaussian 
approximation $\nuposh(y)$ can be written in terms of the generalized eigenvalues of the pencil $(\Gpos,\Gposh)$ as:
\begin{equation*}
D_{\rm KL} \left( \nupos(y) \|\, \nuposh(y) \right)=\sum_{i} \left(\,\sigma_i-\ln\sigma_i-1\,\right)/2,
\end{equation*}
and since $f(x) = \left( x - \ln x -1 \right)/2$ belongs to $\Uc$, we see that the K-L divergence is  a loss function in the class $\Lc$ defined by \eqref{eq:form_loss_fun}. Hence, Theorem \ref{Thm:main_th} applies and the equivalence between the  two approximations follows trivially.
 An analogous argument holds for the Hellinger distance. 
The squared Hellinger distance between $\nu_\mathrm{pos}(y)$ and $\widehat{\nu}_\mathrm{pos}(y)$
can be written in terms of the generalized eigenvalues, $(\sigma_i)$, of the 
pencil $(\Gpos,\Gposh)$, as:
\begin{equation} \label{eq:Hell_dist_eig}
 d_{\mathrm{Hell}}\left( \nu_\mathrm{pos}(y),\widehat{\nu}_\mathrm{pos}(y) \right)^2= 
1 - 2^{\ell/2} \prod_i \sigma_i^{1/4} \left( 1+\sigma_i\right)^{-1/2} .
\end{equation}
where $\ell$ is the dimension of the parameter space.
Minimizing \eqref{eq:Hell_dist_eig} is equivalent to maximizing $\prod_i \sigma_i^{1/4} \left( 1+\sigma_i\right)^{-1/2}$,
which in turn is equivalent to minimizing the functional:
\begin{equation} \label{eq:Hell_fun2min}
L(\Gpos,\Gposh)= -\sum_i \ln(\, \sigma_i^{1/4} ( 1+\sigma_i )^{-1/2}\,)=
\sum_i \ln(\,2 + \sigma_i + 1/\sigma_i\,)/4
\end{equation}
Since $f(x) = \ln(\,2+x+1/x\,)/4$ belongs to $\Uc$,
Theorem \ref{Thm:main_th} can be applied once again. $\square$
 
\medskip\noindent
{\bf Proof of Corollary \ref{cor:Minv}.}
The proofs of parts (i) and (ii) were already given in the proof of Theorem \ref{Thm:main_th}. 
Part (iii) holds because,
\begin{eqnarray*}
 (1+\delta_i^2)\Gpos\, \widetilde{w}_i &=&  (1+\delta_i^2) ( H + \Gpr^{-1} )^{-1}  \Gprs^{-\top}w_i \\
                                    &=&(1+\delta_i^2)\, \Gprs ( \Gprs^{\top}H\,\Gprs + I )^{-1}w_i 
                                    = \Gprs\, w_i 
                                    = \Gpr\, \widetilde{w}_i , 
\end{eqnarray*} 
because $w_i$ is an eigenvector of 
$(\, \Gprs^{\top}H\,\Gprs + I\,)^{-1}$ with eigenvalue $(1+\delta_i^2)^{-1}$ 
as shown in the proof of Theorem  \ref{Thm:main_th}. $\square$

\medskip
Now we turn to optimality results for approximations of the posterior mean. 
In what follows, let $\Gprs$,\, $\Gobss$, \,$\Gposs$, and $\Gys$ be the matrix square roots of, respectively,\, $\Gpr$,\, $\Gobs$, \,$\Gpos$, and $\Gy := \Gobs + G\,\Gpr \,G^\top$ such that $\Gamma=S\,S^\top$ (i.e., possibly non-symmetric square roots).

Equation \eqref{eq:riskdec} shows that, to minimize $\Ex(\,\left\|Ay - x\right\|^2_{\Gpos^{-1}}\,)$ over $A\in \Ac$,
we need only to minimize 
$\Ex(\,\|\,Ay - \mupos\,\|^2_{\Gpos^{-1}}\,)$.
Furthermore, since $\mupos= \Gpos\, G^\top\Gobs^{-1}\,y$, it follows that
\begin{equation}\label{eq:Brisk1}
\Ex(\,\|\,Ay - \mupos\,\|^2_{\Gpos^{-1}}\,) = 
\|\,\Gposs^{-1 }\,(A-\Gpos\, G^\top\Gobs^{-1})\,\Gys \,\|_F^2,
\end{equation}
We are therefore led to the following  optimization problem:
  \begin{equation} \label{eq:optim_mean}
\min_{A\in \Ac} 
\|\,\Gposs^{-1}A\,\Gys -\Gposs^\top G^\top\Gobs^{-1}\,\Gys\,\|_F .
  \end{equation}
 \highlight
{The following result shows that an SVD of the matrix $\Hhs:=\Gprs^\top \,G^\top \Gobss^{-\top}$ 
can be used to obtain simple expressions for the square roots of $\Gpos$ and $\Gy$.}
\medskip
\begin{lemma}[Square roots] \label{lemma:square_roots}
Let $WDV^\top$ be an SVD of $\Hhs=\Gprs^\top \,G^\top \Gobss^{-\top}$.  Then: 
  \begin{eqnarray}  \label{eq:sqroot_covpos}
     \Gposs & = & \Gprs\, W(\,I + D D^\top\,)^{-1/2}\, W^\top \label{eq:sqroot_pos}\\   
     \Gys  &  = &  \Gobss \, V(\,I + D^\top \!D\,)^{1/2}\, V^\top\label{eq:sqroot_marg}
  \end{eqnarray} 
 are square roots of $\Gpos$ and $\Gy$. 
\end{lemma}
\begin{proof}
 We can rewrite $\Gpos  =  (\,G^\top \Gobs^{-1} G + \Gpr^{-1}\,)^{-1}$ as
\begin{eqnarray*}
\Gpos & = & \Gprs\,(\,\Hhs \, \Hhs^\top + I\,)^{-1}\,\Gprs^\top  = \Gprs\,W(\,D D^\top + I\,)^{-1} W^\top\Gprs^\top\\
&=& [\,\Gprs\,W(\,D D^\top + I\,)^{-1/2}\, W^\top\,]\,[\,\Gprs\,W(\,D D^\top + I\,)^{-1/2}\, W^\top\,]^\top,
\end{eqnarray*}
which proves \eqref{eq:sqroot_pos}. The proof of \eqref{eq:sqroot_marg} follows similarly using:
$  \Hhs^\top \,\Hhs = \Gobss^{-1 } G \,\Gpr\, G^\top \Gobs^{-\top}$.
\end{proof}

\smallskip
In the next two proofs we use $\left( C \right)_r$ to denote a rank $r$ approximation of the matrix $C$ in the Frobenius norm.

\smallskip\noindent
{\bf Proof of Theorem \ref{thm:mean_approx_lowrank}.}
By \cite[Theorem 2.1]{friedland2007generalized}, an optimal $A\in \Ac_r$ is given by:
                  \begin{equation} \label{eq:estim_I}
                        A= \Gposs\,\left(\,   \Gposs^\top G^\top\Gobs^{-1}\,\Gys \,\right)_r\,\Gys^{-1 }.
                  \end{equation} 
Now, we need some computations to show that \eqref{eq:estim_I} is equivalent to \eqref{eq: mean_approx_lowrank}. 
Using \eqref{eq:sqroot_pos} and \eqref{eq:sqroot_marg} we find 
\(
 \Gposs^\top G^\top\Gobs^{-1}\,\Gys = W(I+DD^\top)^{-1/2} D\, (I+D^\top D)^{1/2}\,V^\top,
\)
and therefore 
\(
(\,\Gposs^\top G^\top\Gobs^{-1}\,\Gys\,)_r = \sum_{i=1}^r \delta_i w_i v_i^\top
\),
\highlight{
where $w_i$ is the $i$th column of $W$,  $v_i$ is the $i$th column of $V$,
and  $\delta_i$ is the $i$th diagonal entry of $D$.
Inserting this back into \eqref{eq:estim_I} yields 
$A = \sum_{i \le r} \delta_i (1+\delta_i^2)^{-1} \Gprs w_i v_i^\top \Gobss^{-1}$.
Now it suffices to note that $\widehat{w}_i := \Gprs w_i $ is a
generalized eigenvector of $(H,\Gpr^{-1})$, that
$\widehat{v}_i := \Gobss^{-\top} v_i$ is a generalized eigenvector of $(G\Gpr G^\top , \Gobs)$, and that
$(\delta_i^2)$ are also eigenvalues of $(H,\Gpr^{-1})$.
}
The minimum Bayes risk is a straightforward computation for the optimal estimator \eqref{eq: mean_approx_lowrank}
using \eqref{eq:Brisk1}.  $\square$

\medskip\noindent
{\bf Proof of Theorem \ref{thm:mean_approx_fullrank}.}
Given $A\in \widehat{\Ac}_r$, we can restate \eqref{eq:optim_mean} as the problem of finding a matrix $B$, of rank at most $r$, 
that minimizes:
  \begin{equation} \label{eq:estim_prob}
  \|\, \Gposs^{-1 } ( \Gpr - \Gpos )\, G^\top \Gobs^{-1} \,\Gys  - \Gposs^{-1 } B 
\left( G^\top \Gobs^{-1}\, \Gys  \right)\,\|_F
  \end{equation}
such that $A=  \left( \Gpr-B \right)G^\top\Gobs^{-1}$. By \cite[Theorem 2.1]{friedland2007generalized}, an optimal $B$ is given by:
     \begin{equation} \label{eq:estim_II_0}
                        B = \Gposs  (\,  \Gposs^{-1 } \, ( \Gpr - \Gpos )\, 
G^\top \Gobs^{-1}\, \Gys \,)_r ( G^\top \Gobs^{-1}\, \Gys  )^\dagger
     \end{equation}
where $\dagger$ denotes the pseudo-inverse operator. A closer look at
\cite[Theorem 2.1]{friedland2007generalized} reveals that another
minimizer of \eqref{eq:estim_prob}, itself not necessarily  of minimum Frobenius norm, is given by:
   \begin{equation} \label{eq:estim_II}
                        B = \Gposs  (\,  \Gposs^{-1 } \, ( \Gpr - \Gpos )\, 
G^\top \Gobs^{-1}\, \Gys \,)_r \,( \Gprs^\top G^\top \Gobs^{-1}\, \Gys)^\dagger\Gprs^\top .
     \end{equation}
By Lemma \ref{lemma:square_roots},
 \begin{eqnarray*}
 \Gprs^\top G^\top \Gobs^{-1}\, \Gys & = & W [\, D \left( I+D^\top D \right)^{1/2} \,] V^\top \\
\Gposs^{-1 } \,\Gpr G^\top \Gobs^{-1} \Gys  & = & W[\,(I+DD^\top)^{1/2}D(I+D^\top D)^{1/2} \,] V^\top\\ 
\Gposs^{-1 } \,\Gpos G^\top \Gobs^{-1} \Gys  & = &  W[\,(I+DD^\top)^{-1/2}D(I+D^\top D)^{1/2} \,] V^\top
 \end{eqnarray*}
and therefore $ ( \Gprs^\top G^\top \Gobs^{-1}\, \Gys  )^\dagger = \sum_{i=1}^{q} \delta_i^{-1} \left(
 1+\delta_i^2\right)^{-1/2} v_i w_i^\top$ for $q=\mbox{rank}(\Hhs)$, whereas 
\[
(\,\Gposs^{-1 } \, ( \Gpr - \Gpos ) G^\top \Gobs^{-1} \Gys \,)_r = \sum_{i=1}^r \delta_i^3 \,w_i v_i^\top.
\]
\highlight{
Inserting these expressions back into \eqref{eq:estim_II}, we obtain:
 \[
  B=\Gprs  \left( \sum_{i=1}^{r} \frac{\delta_i^2}{1+\delta_i^2}\, w_i w_i^\top \right) \Gprs^\top,
 \]
where $w_i$ is the $i$th column of $W$,
 $v_i$ is the $i$th column of $V$,
and  $\delta_i$ is the $i$th diagonal entry of $D$.
Notice that $(\delta_i^2,\widehat{w}_i)$, with $\widehat{w}_i=\Gprs w_i$, are the 
generalized eigenpairs of $(H,\Gpr^{-1})$ (cf.\ proof of Theorem \ref{Thm:main_th}).
Hence, by Theorem \ref{Thm:main_th}, we recognize 
the optimal approximation of $\Gpos$ as $\Gposh=\Gpr-B$.
 Plugging this expression back into \eqref{eq:estim_II} gives \eqref{eq: mean_approx_fullrank}. The 
minimum Bayes risk in (ii) follows readily using the optimal estimator given by                
\eqref{eq: mean_approx_fullrank} in \eqref{eq:Brisk1}. $\square$  
}

 \bibliographystyle{siam}
\bibliography{linearbib}

\end{document}